\documentclass[11pt,twoside]{article}
\usepackage{amsmath, amsfonts, amsthm, amssymb, amscd,mathrsfs}

\usepackage{bbm}
\usepackage{appendix}

\usepackage{color}

\usepackage{hyperref}
\makeatletter
\newdimen\rh@wd
\newdimen\rh@hta
\newdimen\rh@htb
\newbox\rh@box
\def\rh@measure#1{\setbox\rh@box=\hbox{$#1$}\rh@wd=\wd\rh@box \rh@hta=\ht\rh@box}

\def\widecheck#1{\rh@measure{#1}%
  \setbox\rh@box=\hbox{$\widehat{\vrule height \rh@hta width\z@ \kern\rh@wd}$}%
  \rh@htb=\ht\rh@box \advance\rh@htb\rh@hta \advance\rh@htb\p@
  \ooalign{$\vrule height \ht\rh@box width\z@ #1$\cr
           \raise\rh@htb\hbox{\scalebox{1}[-1]{\box\rh@box}}\cr}}
\makeatother

\numberwithin{equation}{section}  
\usepackage{perpage} 
\MakePerPage{footnote}  
\oddsidemargin 	.in
\evensidemargin.in
\topmargin    - 0.75in
\numberwithin{equation}{section}
        \newtheorem{theorem}{Theorem}[section]
        \newtheorem{proposition}[theorem]{Proposition}
        \newtheorem{lemma}[theorem]{Lemma}
        \newtheorem{corollary}[theorem]{Corollary}

        \newtheorem{remark}[theorem]{Remark}  
\let\oldmarginpar\marginpar
\renewcommand\marginpar[1]{\-\oldmarginpar[\raggedleft\footnotesize #1]
{\raggedright\footnotesize #1}}

\newcommand \bei {\begin{itemize}}

\newcommand \eei {\end{itemize}} 
\newcommand \Hcal {\mathcal H}

\newcommand  \rd {{\rm d}}

\newcommand \del \partial
\newcommand \underdel {\underline \partial}

\newcommand \be {\begin{equation}}
\newcommand \bel {\be\label}
\newcommand \ee {\end{equation}}
\newcommand \la \langle
\newcommand \ra \rangle 
\newcommand	\Kcal 	{\mathcal K}

\newcommand	\RR 		{\mathbb R}  
  
\newcommand  \CC        {\mathbb C}

\newcommand \eps \epsilon

\newcommand \Ecal {\mathcal E}

\begin{document}

\title{Global solution to the cubic Dirac equation in  two space dimensions} 

\author{Shijie Dong$^{\dag}$,\ \ Kuijie Li$^{\ddag,}$\footnote{K. Li is the corresponding author. 
Email: shijiedong1991@hotmail.com, kuijiel@nankai.edu.cn
} \\
\footnotesize \it $^\dag$School of Mathematical Sciences, Fudan University, Shanghai, 200433, China \\
\footnotesize \it $^\ddag$School of Mathematical Sciences, Nankai University, Tianjin, 300071, China
}

\date{\today}

\maketitle

\begin{abstract} 

We are interested in the cubic Dirac equation with mass $m \in [0, 1]$ in two space dimensions, which is also known as the Soler model. We conduct a thorough study on this model with initial data sufficiently small in high regularity Sobolev spaces. First, we show the global existence of the model, which is uniform-in-mass. In addition, we derive a unified pointwise decay result valid for all $m \in [0, 1]$. Last but not least, we prove the cubic Dirac equations scatter linearly with an explicit scattering speed. When the mass $m=0$, we can show an improved pointwise decay result. 
\\



\end{abstract} 
\maketitle 

\tableofcontents


\

\section{Introduction}
\label{sec:1}

\paragraph{Model problem}

We are interested in the following nonlinear Dirac equation in $\RR^{1+2}$ 
\begin{equation} 
i\gamma^\mu \del_\mu \psi + m \psi = F(\psi), 
\ee
with prescribed initial data on $t=t_0=2$
\bel{eq:ID-Dirac}
\psi(t_0,x) = \psi_0(x).
\end{equation}
In the above, the spinor $\psi(t,x): \RR^{1+2} \to \mathbb{C}^2$,  the Dirac operator $i\gamma^\mu \del_\mu = i\gamma^0 \del_t + i\gamma^1 \del_1 + i\gamma^2 \del_2$ with $\gamma^\mu$ the Gamma matrices satisfying ($I_2$ the $2\times 2$ identity matrix)
\be  \label{gammamatrice}
\gamma^\mu \gamma^\nu + \gamma^\nu \gamma^\mu = -2 g^{\mu\nu} I_2, 
\qquad (\gamma^\mu)^* = -g_{\mu\nu} \gamma^\nu,
\ee
and $g = \text{diag}(-1, 1, 1)$ is used to denote the Minkowski metric in $\RR^{1+2}$. The mass parameter $m$ is assumed to lie in $[0, 1]$. Here  $\mu, \nu, \cdots \in \{0, 1, 2\}$ represent the space-time indices, and the Einstein summation convention is adopted. 
We will consider  a special case of \eqref{eq:model-Dirac} where $F$ is cubic and has some additional structure. More precisely, we set 
\begin{equation} \label{eq:nonlinearity}
F(\psi) = \begin{cases}
(\psi^* \gamma^0 \psi) \psi, \\
(\psi^*\gamma^0\gamma^{\mu} \psi)\gamma_{\mu} \psi, 
\end{cases}
\end{equation}
which are known as the Soler model \cite{So70} and the Thirring model \cite{Th58}, respectively. Recall that $\psi^*$ represents the complex conjugate transpose of the vector $\psi$ and $\gamma_{\mu} = g_{\mu\nu}\gamma^{\nu}$. In many literature, $\bar{\psi}:=\psi^* \gamma^0$ is used to denote the Dirac adjoint. There is an interesting phenomenon, as observed by the authors in \cite{BoCa16},  that the above two types of nonlinearities in \eqref{eq:nonlinearity} coincide in two space dimensions. {For this reason}, we will focus on the study of the  Dirac equation of Soler type, i.e. 
\begin{align} \label{eq:model-Dirac}
i\gamma^\mu \del_\mu \psi + m \psi  = (\psi^* \gamma^0\psi)\psi, \ \ \   \psi(t_0,x) = \psi_0(x).
\end{align}

The nonlinear Dirac equation has essential applications in relativistic quantum mechanics and models the self-interaction of Dirac fermions. For more information of its physical background, one refers to \cite{Th92}. 

Recall that the following relation holds
\be 
\big( i\gamma^\mu \del_\mu  \big) \big( i\gamma^\nu \del_\nu  \big) = \Box,  \nonumber
\ee
in which $\Box = g^{\alpha\beta} \del_\alpha\del_\beta = -\del_t \del_t + \Delta$ is the wave operator. Applying the Dirac operator $i\gamma^{\nu}\partial_{\nu}$ to both sides of \eqref{eq:model-Dirac}, one can easily find  $\psi$ also solves the following Klein-Gordon(or wave, if $m=0$) equation
\begin{align*}
-\Box \psi + m^2 \psi = m F(\psi) - i\gamma^{\mu}\partial_{\mu} F(\psi).
\end{align*}
Many properties of wave equations such as finite propagation speed and local existence results are also enjoyed by the Dirac equations. On the other hand, in terms of the decay rate, the linear wave solution $u$ and linear Klein-Gordon solution $v$ decay relatively slowly in two space dimensions, indeed, 
\[
|u(t,x)| \lesssim (1+t+|x|)^{-\frac{1}{2}} (1+|t-|x||)^{-\frac{1}{2}}, \qquad |v(t,x)| \lesssim (1+t+|x|)^{-1}.
\]
This indicates that the Dirac solution $\psi$ should decay like $u$ ($m=0$) or  $v$ ($m\neq 0$).

In case of $m=0$, the nonlinear Dirac equation with cubic nonlinearity enjoys a good scaling structure. Particularly, if $\psi$ solves \eqref{eq:model-Dirac}, so does $\psi_{\lambda}(t,x):= \lambda^{\frac{1}{2}} \psi(\lambda t,\lambda x)$. A function space $X$ is called a critical space for Dirac equation provided $ \|\psi_{\lambda}(t_0,x)\|_{X} = \|\psi(t_0,x)\|_{X}$ For instance, the space $\dot{H}^{\frac{d-1}{2}}(\RR^d)$ is a critical space for the $d$-dimensional Dirac equation. Empirically, we expect the Dirac equations to be well-posed(at least locally in time) for data in $H^s$ with $s\geq \frac{d-1}{2}$ and ill-posed if $s<\frac{d-1}{2}$. 

There is an extensive literature on the study of the well-posedness of the Cauchy problem to the Dirac equation. Concerning the low regularity setting,  global well-posedness is showed for small $H^s(\RR^3)$ data with $s>1$ and  positive mass $m>0$ and cubic nonlinearity in \cite{MaNaOz03}.  This result was improved for Dirac equations with critical $H^1(\RR^3)$ radial initial data or data with some additional angular regularity, see  \cite{MaNaNaOz05}. Recently, global well-posedness and scattering are obtained for solutions to the massive Soler model with small data in the critical space $H^1(\RR^3)$ by Bejenaru-Herr \cite{BeHe15}. In the low dimensional case, Pecher \cite{Pe14} proved the local well-posedness for initial data in almost critical space $H^s(\RR^2)$ with $s>1/2$, this result was further extended by Bournaveas and Candy \cite{BoCa16}, where the authors  proved the global existence, uniqueness and scattering for both Soler and Thirring model   with $m=0$ and initial data in critical space $\dot{H}^{\frac{d-1}{2}}(\RR^d)$, where $d=2,3$. 

We are interested in the global existence and long time asymptotic behavior for the solution to the Dirac equation with regular small initial data in two space dimensions. In this direction, the work of
 Escobedo and Vega \cite{EsVe97} proved the well-posedness of $\RR^{1+3}$ Dirac equation with homogeneous nonlinearity $F(\psi)$ of degree $p\geq 3$, in particular, local well-posedness holds for the cubic Dirac equation with data in $H^s(\RR^3), \,s>1$. Further extension was made by Tzvetkov \cite{Tz98}, where he showed the global existence result for $\RR^{1+3}$ massless Dirac equation with small regular Sobolev data and growth condition $|F(\psi)| \lesssim |\psi|^p,\, p>2$. {We also recall the relevant results in \cite{Yin,CaLeMa18,He2020}.} We will work on the $\RR^{1+2}$ case, and our main results are as follows.
 
\paragraph{Main results}

\begin{theorem}[Global existence] \label{thm:mainresult}
Let $N \geq 2 $ be an integer, and consider the Dirac equation \eqref{eq:model-Dirac} with initial data $\psi_0$ on the time slice $t_0 = 2$ supported in the unit ball $\{  x\in \mathbb{R}^2 : |x| \leq 1  \}$. 
Then there exists $\epsilon_0 > 0$, which  in particular is independent of the mass parameter $m\in [0, 1]$, such that, for all $0<\eps \leq \eps_0$ and all initial data satisfying
\be 
\| \psi_0 \|_{H^{N+1}(\RR^2)}
\leq \eps,   \nonumber
\ee
the Cauchy problem \eqref{eq:model-Dirac}--\eqref{eq:ID-Dirac} admits a unique global-in-time solution $\psi$, which
 satisfies the unified decay result
\bel{eq:solution-decay}
|\psi(t, x)| \lesssim {\epsilon \over t^{1/2}(t-r)^{1/2} + mt},
\ee
where $r=|x|$ and $ t\geq t_0.$
\end{theorem}

\begin{remark}
Compared with the previous work, we can obtain the global existence result for \eqref{eq:model-Dirac} which is uniform in terms of the mass parameter $m \in [0, 1]$, and we can also get the decay result \eqref{eq:solution-decay} for the solution which unifies all the cases of $m \in [0, 1]$.
\end{remark}

 From Theorem \ref{thm:mainresult}, we know that the solution to the cubic Dirac equation in two space dimensions \eqref{eq:model-Dirac} exists globally for small localized initial data, and the pointwise decay result holds. On the other hand, as the initial data belongs to the Sobolev space, it is also natural to measure the solution in the same space, especially for its long time behavior. Our following result shows that the solution to the cubic Dirac equation scatters (linearly) as time tends to infinity with a scattering speed.  

\begin{theorem}[Scattering result] \label{thm:scatterresult}
 Let the assumptions in Theorem \ref{thm:mainresult} hold,  then there exists some  $\epsilon_1\leq \epsilon_0$, such that the  global solution $\psi$ associated with $\psi_0$ scatters as $t\to +\infty$, provided  
 \[
 \|\psi_0\|_{H^{N+1}} \leq \epsilon \leq \epsilon_1.
 \]  
 More precisely, there exists some $\psi^+ \in H^{N+1}$ and constant $C$, such that 
 \begin{align} \label{eq:scatterhighorder}
 \|\psi(t) - S(t-t_0) \psi^+\|_{H^{N+1}} \leq Ct^{-1/2}, \ \ \forall\ t \geq t_0.	
 \end{align}
 Furthermore, there exists some $C(t)>0$ with $\lim_{t\to +\infty} C(t) =0$, satisfying 
 \begin{align} \label{eq:scatterloworder}
 	\|\psi(t) - S(t-t_0) \psi^+\|_{H^{N-1}} \leq C(t) t^{-1/2}, \ \ \forall\ t \geq t_0.
 \end{align}
 where $S(t):= e^{it(i\gamma^0\gamma^a \partial_a + m\gamma^0)}$ is the propagator for the linear Dirac equation.
 
  \end{theorem}
  
  \begin{remark}
  	In contrast with the scattering result in \cite{BoCa16}, we obtain additionally the explicit scattering speed, and our result applies to both massive and massless  Dirac equations.
  \end{remark}
In the case $m=0$, as observed by Bournaveas in \cite{Bo01}, the Dirac solution $\psi$ can be reformulated as the Dirac operator acting on some linear wave solution $\Psi$ by setting $i\gamma^{\mu}\partial_{\mu} \Psi = \psi$. One advantage for treating the wave solution $\Psi$ in our analysis is that a conformal energy estimate is available so that we can get a better decay result for the differentiated wave solution $\partial \Psi$ and hence $\psi$. Such observation motivates us to improve the $(t-r)$-decay for the Dirac solution, and this is indeed the case, at the cost of some logarithmic growth in time $t$. We remark that the following result improves the one in Theorem \ref{thm:mainresult} significantly when $r\leq t/2$ for instance.

\begin{theorem} \label{thm:improveddecay}
Let $m=0$ and let the assumptions in Theorem \ref{thm:mainresult} hold. Suppose that $\psi$ is the global solution to Dirac equation \eqref{eq:model-Dirac} given by Theorem \ref{thm:mainresult}, then
 	\begin{align*}
		|\psi(t,x)| \lesssim  \min \Big\{\frac{1}{t^{1/2}(t-r)^{1/2}}, \, \frac{(\ln t)^2 }{t^{1/2}(t-r)^{3/2}} \Big\} \epsilon.
	\end{align*}  
\end{theorem}

We note that the main contribution of Theorem \ref{thm:improveddecay} is that we get a pointwise decay for the massless cubic Dirac equation with an improved homogeneity (the total power of the decay rate), i.e., from $-1$ in Theorem \ref{thm:mainresult} to almost $-2$ in Theorem \ref{thm:improveddecay}. There exist analogous studies on improving the homogeneity of the pointwise decay of the Dirac equation, which are also an interesting and challenging subject; see, for instance, the work \cite{MaZh20} and the references therein.


\paragraph{Strategy of proof}

We now very briefly demonstrate the key ideas used to show the main theorems above, and one finds more details in the proofs.

We will rely on the hyperboloidal foliation of the space-time to prove the existence of a global solution and the unified pointwise time-decay result.
The advantage is that the scaling vector field $L_0 = t\del_t + x^a\del_a$ can be avoided in the analysis, and thus we can treat all the cases $m\in [0, 1]$ in one framework. Furthermore we can get a mass-dependent decay result as well as a wave decay result with the implicit constant independent of the mass.

In general, cubic nonlinearities cannot guarantee the small data global existence result for two-dimensional wave equations. However, we will explore some special structures of the nonlinearities so that we can get enough decay to close the bootstrap argument.

Concerning the long time scattering result, we come back to the constant time foliation and make use of the classical ghost weight energy method introduced by Alinhac in \cite{Alinhac1}. In order to close the estimate, the pointwise decay result derived using the hyperboloidal foliation method is fully used.

\paragraph{Outline} The rest of the paper is organized as follows. In Section \ref{sec:intro}, we mainly present some preliminaries on the hyperboloidal foliation method, recall the energy estimates of the Dirac equation on hyperboloids and show the hidden structure of the nonlinear term.  Section \ref{sec:conclusion} is devoted to proving the global existence and the pointwise decay. We demonstrate the scattering result in Section \ref{ch:scattering}.  Finally, the proof of Theorem \ref{thm:improveddecay} is put in the last section.

\section{Preliminaries} \label{sec:intro}

\subsection{Notations}

We adopt the hyperboloidal foliation method \cite{LeMa14,LeMa16} by LeFloch-Ma, which takes its root in \cite{Kl85} by Klainerman and \cite{Ho97} by H\"ormander, to prove the main result. We first briefly recall some notations from \cite{LeMa14,LeMa16}.  We denote a space-time point in $\RR^{1+2}$ by $(t,x)=(x_0, x_1, x_2)$ with $t=x_0, \, x=(x_1,x_2)$, $x^a=x_a,\, a=1,2$, and its spacial radius is denoted by $r:=|x|=\sqrt{x_1^2 + x_2^2}$. Following the notations in \cite{Klainerman80}, we use
$$
\partial_{\alpha}:=\partial_{x_{\alpha}}, \ \ L_a:=t\partial_a + x_a\partial_t, \ \  \Omega_{12}:=x_1\partial_2-x_2\partial_1,\ \ \textrm{and}\ \ L_0:= t\partial_t + x_a \partial_a 
$$
to represent the translations, the Lorentz boosts, the rotation and the scaling vector field respectively. Einstein summation convention over repeated indices is adopted, e.g. $x^a \partial_a = x^1\partial_1+x^2\partial_2$. In the sequel, Roman letters indices $a, b \in \{1,2\}$, Greek letters indices $\alpha, \beta \in \{0, 1, 2\}$, if there is no otherwise indications. As usual, given a vector or a scalar $w$ we use Japanese bracket to denote $\langle w \rangle := (1+|w|^2)^{1/2}$.

We remark that the Lorentz boosts $L_a$ do not directly commute with the Dirac operator $i\gamma^{\mu} \partial_{\mu}$. Following \cite{Ba88} by Bachelot, we introduce the modified Lorentz boosts:
\begin{align*}
\widehat{L}_a := L_a -\frac{1}{2} \gamma^0\gamma^a.
\end{align*}
The key feature of $\widehat{L}_a$ is that it commutes with the Dirac operator, i.e.
\[
[\widehat{L}_a, i\gamma^{\mu}\partial_{\mu}] =0,
\]
where we use the standard notation for the commutator $[A,B]:= AB-BA$.

From now on, we consider functions supported in  the interior of the  light cone $\Kcal:= \{(t,x):t\geq 2, |x|<t-1 \}$. A hyperboloid $\Hcal_s$ with hyperbolic time $s\geq s_0 =2$ is defined as 
\[
\Hcal_s:= \{(t,x): t^2 = |x|^2 + s^2\}. 
\]
It is not hard to find that for any point $(t,x)\in \Kcal \cap \Hcal_s$ with $s\geq 2$, we have 
\[
|x| \leq t, \ \ \ \ s\leq t\leq s^2.
\]
Also we use $\Kcal_{[s_0,s_1]}: = \bigcup_{s_0\leq s\leq s_1} \Hcal_{s}\cap \Kcal$ to denote the space-time region between two hyperboloids $\Hcal_{s_0}, \Hcal_{s_1}$.
For a (vector-valued) function $f$ defined in $\RR^{1+2}$, we denote 
\[
\|f\|^p_{L^p_f(\Hcal_s)} =\int_{\Hcal_s} |f|^p\, {\rm d}x := \int_{\RR^2} |f(\sqrt{s^2+|x|^2},x)|^p\,{\rm d}x, \qquad  1\leq p<\infty.
\]
Following the work of LeFloch and Ma \cite{LeMa14}, we introduce the semi-hyperboloidal frame
\[
\underline{\partial}_0 : =\partial_t, \ \ \quad \underline{\partial}_a: = \frac{L_{a}}{t}=\frac{x_a}{t} \partial_t + \partial_a.
\]
It is worth mentioning that $\{\underline{\partial}_1,\underline{\partial}_2\}$ generates the tangent space to the hyperboloid $\Hcal_s$.

Throughout this paper, $C$ denotes a generic constant, which may be different from line to line. As usual, $A \lesssim B$ means that $A \leq C B$ for some constant $C$, while $A\sim B$ means $A \lesssim B$ and $B\lesssim A$. We use $\mathcal{F}$ and $\mathcal{F}^{-1}$ to denote the Fourier transform and inverse Fourier transform respectively. 
We also denote the following  ordered set by
$$
\{\Gamma_i\}_{i=1}^7:= \big\{\partial_0,\partial_1,\partial_2, L_1, L_2,\widehat{L}_1, \widehat{L}_2 \big\}.
$$

For an arbitrary multi-index $I=(\alpha_1,\alpha_2,\alpha_3)$ of length $|I|=\alpha_1 + \alpha_2 + \alpha_3=n$, we use $\partial^I:=\Gamma_1^{\alpha_1} \Gamma_2^{\alpha_2} \Gamma_3^{\alpha_3}$ to denote the $n$-th order derivatives. Similarly we represent $L^{J}:= \Gamma_4^{\beta_1} \Gamma_5^{\beta_2}$ as well as $\widehat{L}^{J}:= \Gamma_6^{\beta_1} \Gamma_7^{\beta_2}$, where $J = (\beta_1, \beta_2)$.


\subsection{Energy estimates for Dirac field on hyperboloids}

Let us recall the energy estimates for the Dirac field on hyperboloids $\Hcal_s$, which was first introduced in \cite{DoLeWy21} (see Section 2.3). Given a complex-valued function $\psi(t,x):\RR^{1+2} \to \mathbb{C}^2$ on $\Hcal_s$, we define the following energy functional
\begin{align*}
\Ecal^{D}(s,\psi):&= \int_{\Hcal_s} \Big(\psi^*\psi-\frac{x_a}{t} \psi^* \gamma^0\gamma^a \psi \Big)\, \rd x,  \\
\Ecal^{+}(s,\psi):&= \int_{\Hcal_s} \Big(\psi- \frac{x_a}{t} \gamma^0\gamma^a \psi \Big)^* \Big(\psi- \frac{x_a}{t} \gamma^0 \gamma^a \psi  \Big)\,\rd x.
\end{align*}
An important relation  between $\Ecal^D(s,\psi)$ and $\Ecal^+(s,\psi)$ was found in \cite{DoLeWy21} (see Proposition 2.3),  which reads
\begin{align*}
\Ecal^D(s, \psi)= \frac{1}{2} \int_{\Hcal_s} \frac{s^2}{t^2} \psi^* \psi\, \rd x + \frac{1}{2} \Ecal^+(s,\psi).
\end{align*} 
Hence, the positivity of the energy functional $\Ecal^{D}(s,\psi)$ follows as a direct result and 
\begin{align} \label{energybound}
\Big\|\frac{s}{t} \psi \Big\|_{L_f^2(\Hcal_s)}+ \Big\|\Big( I_2- \frac{x_a}{t}\gamma^0\gamma^a\Big) \psi \Big\|_{L_f^2(\Hcal_s)} \leq 4\, \Ecal^D(s,\psi)^{\frac{1}{2}}.
\end{align}
Finally, let us illustrate the energy estimates (see \cite{DoLeWy21}, Proposition 2.4), which shall play a crucial role in our proof.
\begin{proposition} \label{boundsofenergy}
Let $\psi(t, x): \RR^{1+2} \to \mathbb{C}^2$ be a sufficiently regular function with support in the region $\Kcal_{[s_0,s_1]}$ and $m\in \RR$ be arbitrary. Then for all $s\in [s_0,s_1]$, we have 
\begin{align*}
\Ecal^{D}(s,\psi)^{1/2} \leq \Ecal^D(s_0,\psi)^{1/2} + \int_{s_0}^s \|i\gamma^{\mu}\partial_{\mu} \psi + m\psi\|_{L_f^2(\Hcal_{\tau})} \,\rd \tau.
\end{align*}
\end{proposition}

\begin{remark}
It should be noted that $'m'$ does not appear in the formulation of $\Ecal^D(s,\psi)$. Indeed, there is a delicate cancellation in the derivation of this energy estimates.
\end{remark}

\subsection{Estimates for commutators and Sobolev inequalities}
The following commutator estimates will be frequently used and the proof is straightforward, for instance one can refer to \cite{LeMa14}, Chapter 3.
\begin{lemma} \label{lem:commu}
Let $\Phi, \phi$ be a sufficiently regular $\CC^2$-valued (resp. $\RR$-valued) function supported in the region $\mathcal{K}$. Then, for any multi-indices $I,J$, there exist generic constants $C=C(|I|, |J|)>0$ such that
$$
\aligned
&\big| [\del_\alpha, L_a] \Phi \big| + \big| [\del_\alpha, \widehat{L}_a] \Phi \big|
\leq C |\del \Phi|,
\\
&\big| [L_a, L_b] \Phi  \big| + \big| [\widehat{L}_a, \widehat{L}_b] \Phi  \big|
\leq C \sum_c |L_c \Phi|,
\\
& \big| [\del^I L^J, \del_\alpha] \phi \big| 
\leq 
C \sum_{|J'|<|J|} \sum_\beta \big|\del_\beta \del^I L^{J'} \phi \big|,
\\
& \big| [\del^I L^J, \underdel_a] \phi \big| 
\leq 
C \Big( \sum_{| I' |<| I |, | J' |< | J |} \sum_b \big|\underdel_b \del^{I'} L^{J'} \phi \big| + t^{-1} \sum_{| I' |\leq | I |, |J'|\leq |J|} \big| \del^{I'} L^{J'} \phi \big| \Big).
\endaligned
$$
Recall that Greek indices $\alpha, \beta \in \{0,1,2\}$ and Roman indices $a,b \in \{1,2\}$. 
\end{lemma}

Finally, let us state the well-known Sobolev inequality on hyperboloids, which is due to Klainerman \cite{Kl85}, then refined by H\"ormander \cite[Lemma 7.6.1]{Ho97} and simplified by LeFloch-Ma \cite{LeMa14}. The following version is taken from \cite{LeMa14}, Chapter 5. 
\begin{proposition} \label{prop:sob}
Let $ \phi(t,x):\RR^{1+2} \to \mathbb{C}$ be a sufficiently smooth function supported in the region $\Kcal$. Then for all $s\geq 2$, we have
\begin{align*}
  \sup_{\Hcal_s} |t \phi(t,x)| \leq C \sum_{|J|\leq 2} \|L^J \phi\|_{L_f^2(\Hcal_s)}.
\end{align*}
\end{proposition}

\begin{corollary} \label{cor:weightl2estimate}
Let $ \phi(t,x):\RR^{1+2} \to \mathbb{C}$ be a sufficiently smooth function supported in the region $\Kcal$. Then for all $s\geq 2$, we have 
\begin{align*}
  \sup_{\Hcal_s} |s \phi(t,x)| \leq C \sum_{|J|\leq 2} \Big\|\frac{s}{t}L^J \phi \Big\|_{L_f^2(\Hcal_s)}.
\end{align*}
\end{corollary}
\begin{proof}
Set $\tilde{\phi}(t,x)= \frac{s}{t}\phi(t,x)$. By a direct calculation, one can find  
\begin{align*}
\Big[L_a, \frac{s}{t} \Big] &= -\frac{x^a}{t} \frac{s}{t}, \\
\Big[L_bL_a, \frac{s}{t} \Big] & = \Big(2 \frac{x^ax^b}{t^2}- \delta_a^b \Big) \frac{s}{t} -\frac{x^b}{t} \frac{s}{t}L_a,
\end{align*}
where $\delta_{ab}$ denotes the Kronecker symbol, $\delta_{aa} =1$ and $\delta_{ab} =0$ if $b\neq a$.
Applying the above commutator estimates and Proposition \ref{prop:sob} to $\tilde{\phi}$, one can obtain the desired result.
\end{proof}

Noticing the difference between $\widehat{L}_a$ and $L_a$ is a constant matrix for each $a=1,2$, we also have the following.
\begin{corollary} \label{cor:sobolev}
Let $ \psi(t,x):\RR^{1+2} \to \mathbb{C}^2$ be a sufficiently smooth vector valued function supported in the region $\Kcal$. Then for all $s\geq 2$, we have 
\begin{align*}
  \sup_{\Hcal_s} |t \psi(t,x)| \leq C \sum_{|J|\leq 2} \|\widehat{L}^J \psi\|_{L_f^2(\Hcal_s)},
\end{align*}
and 
\begin{align*}
  \sup_{\Hcal_s} |s \psi(t,x)| \leq C \sum_{|J|\leq 2} \Big\|\frac{s}{t}\widehat{L}^J \psi \Big\|_{L_f^2(\Hcal_s)}.
\end{align*}
\end{corollary}


  \subsection{Hidden structure for the nonlinear term } \label{sec:hiddenstruc}
To perform the  energy estimate for the Dirac equation, we will need to treat terms like 
$\|\psi^* \gamma^0 \psi\|_{L_f^2(\Hcal_s)}$ and $\|\frac{t}{s} \psi^* \gamma^0 \psi\|_{L^{\infty} (\Hcal_s)}$.  If we ignore $\gamma^0$, a direct estimate gives us
\[
\|\psi^* \gamma^0 \psi\|_{L_f^2(\Hcal_s)} \leq  \Big\|\frac{t}{s} \psi  \Big\|_{L^{\infty}(\Hcal_s)}  \Big\|\frac{s}{t} \psi \Big\|_{L_f^2(\Hcal_s)}.
\]
However, a favorable bound for  $\big\|(t/s) \psi  \big\|_{L^{\infty}(\Hcal_s)}$ is not available. Take $\gamma^0$ into consideration, and fortunately, we find the nonlinear term $\psi^* \gamma^0 \psi$ possesses a good structure,  which was exploited in \cite{DoWy21}. Let us denote 
\begin{align*}
(\Psi)_+ = \Psi + \frac{x_a}{t} \gamma^0 \gamma^a \Psi, \ \quad \ (\Psi)_{-}= \Psi - \frac{x_a}{t} \gamma^0 \gamma^a \Psi.
\end{align*}  
Now we can state the lemma proved in \cite{DoWy21}, and to make the paper more self-contained, the proof will also be demonstrated. 
\begin{lemma} \label{lem:struc}
Let $\Psi, \Phi$ be two $\mathbb{C}^2$-valued functions, then 
\begin{align*}
\Psi^*\gamma^0\Phi = \frac{1}{4} \Big((\Psi)_{-}^* \gamma^0 (\Phi)_{-} + (\Psi)_{-}^* \gamma^0 (\Phi)_+ + (\Psi)_+^* \gamma^0 (\Phi)_{-} + \frac{s^2}{t^2} \Psi^*\gamma^0 \Phi  \Big).
\end{align*}
\end{lemma}
\begin{proof}
It is obvious that $\Psi = \frac 12\big( (\Psi)_+ +(\Psi)_{-} \big)$ and $\Phi = \frac 12\big( (\Phi)_{+} +(\Phi)_{-} \big)$. Hence, we see
\begin{equation} \label{equality:1}
\begin{split} 
\Psi^*\gamma^0\Phi = \frac{1}{4} \Big((\Psi)_{-}^* \gamma^0 (\Phi)_{-} + (\Psi)_{-}^* \gamma^0 (\Phi)_+ + (\Psi)_+^* \gamma^0 (\Phi)_{-} + (\Psi)_{+}^*\gamma^0 (\Phi)_{+}  \Big).
\end{split}
\end{equation}
On the other hand, in view of \eqref{gammamatrice}, one can find 
\begin{align*}
(\gamma^0)^2 =I_2, \ \ & (\gamma^a)^2 =-I_2, \ \  (\gamma^0)^* = \gamma^0, \ \ (\gamma^a)^* = - \gamma^a, \ \ a\in \{1,2\}, \\
& \gamma^{\mu} \gamma^{\nu} + \gamma^{\nu}\gamma^{\mu} = 0, \ \ \mu \neq \nu, \ \mu, \nu \in \{0, 1, 2 \}.
\end{align*}
Now we can calculate $(\Psi)_{+}^*\gamma^0 (\Phi)_{+}$, 
\begin{align*}
&(\Psi)_{+}^*\gamma^0 (\Phi)_{+} \\
& = \Big(\Psi^* + \frac{x_a}{t} \Psi^* (\gamma^a)^* (\gamma^0)^*\Big) \gamma^0 \Big( \Phi + \frac{x_b}{t} \gamma^0 \gamma^b \Phi \Big) \\
& = \Psi^* \gamma^0 \Phi + \frac{x_b}{t}\Psi^* \gamma^0 \gamma^0\gamma^b \Phi + \frac{x_a}{t} \Psi^*(\gamma^a)^*(\gamma^0)^* \gamma^0 \Phi + \frac{x_ax_b}{t^2}\Psi^* (\gamma^a)^* (\gamma^0)^*\gamma^0\gamma^0\gamma^b \Phi \\
& = \Psi^* \gamma^0 \Phi + \frac{x_b}{t}\Psi^* \gamma^b \Phi-\frac{x_a}{t} \Psi^*\gamma^a \Phi +\frac{x_ax_b}{t^2}\Psi^* \gamma^0\gamma^a\gamma^b \Phi  \\
& = \Psi^* \gamma^0 \Phi + \frac{x_a^2}{t^2}\Psi^* \gamma^0\gamma^a\gamma^a \Phi + \frac{x_1x_2}{t^2}\Psi^* \gamma^0(\gamma^1\gamma^2 + \gamma^2\gamma^1)\Phi \\
&= \frac{s^2}{t^2} \Psi^* \gamma^0 \Phi.
\end{align*}
 Inserting this equality into \eqref{equality:1}, we get the desired result. The proof is complete.
\end{proof}

The following lemma will also be used in the sequel, whose proof is essentially contained in \cite{DoWy21}. 

\begin{lemma} \label{modifiedlorentz}
For arbitrary multi-indices $I, J$, there exists a constant $C=C(I, J)$, such that 
\begin{align*}
|\partial^IL^{J} (\psi^* \gamma^0 \psi)| \leq C \sum_{\substack{|I_1|+|I_2|\leq |I| \\|J_1|+|J_2|\leq |J|} } |(\partial^{I_1}\widehat{L}^{J_1}\psi)^*\gamma^0 \partial^{I_2}\widehat{L}^{J_2}\psi|.
\end{align*}
\end{lemma}
\begin{proof}
We repeat the proof for completeness. For any $\mathbb{C}^2$ valued functions $\Phi, \Psi $, one can see via a direct computation that
\[
L_a(\Phi^*\gamma^0\Psi) = (L_a \Phi)^* \gamma^0 \Psi+ \Phi^* \gamma^0 L_a \Psi.
\]
As an operator acting on $\mathbb{C}^2$-valued function, we have  $L_a = \widehat{L}_a + \frac 12 \gamma^0 \gamma^a$. Hence 
\begin{align*}
L_a(\Phi^*\gamma^0\Psi) & = (\widehat{L}_a \Phi)^* \gamma^0 \Psi+  \frac 12 \Phi^* (\gamma^a)^* (\gamma^0)^* \gamma^0 \Psi +  \Phi^* \gamma^0 \widehat{L}_a \Psi  + \frac  12 \Phi^*\gamma^0 \gamma^0 \gamma^a \Psi \\
& =  (\widehat{L}_a \Phi)^* \gamma^0 \Psi+ \Phi^* \gamma^0 \widehat{L}_a \Psi,
\end{align*}
where we used \eqref{gammamatrice} in the last equality. Thus 
\[
L_a(\psi^* \gamma^0 \psi) = (\widehat{L}_a \psi)^* \gamma^0 \psi + \psi^* \gamma^0 \widehat{L}_a \psi.
\]
In the same way, the Leibniz rule for the derivative operator $\partial$yields the desired result. The proof is done.
\end{proof}

\section{Proof of Theorem \ref{thm:mainresult}} \label{sec:conclusion}

\subsection{Bootstrap assumptions and some direct estimates}
Let $N\in \mathbb{N}$ be an integer ($N\geq 2$ will suffice for our argument). Following the local well-posedness theory in \cite{LeMa14}, Chapter 11, a solution $\psi$ evolving from spacial localized initial data $\psi_0$ can exist up to the initial hyperboloid $\{s=s_0 \}$ with the smallness conserved, which means there exists $C_0>0$ such that 
\begin{equation} \label{initialcondition}
\Ecal^D(s_0, \partial^I \widehat{L}^J \psi)^{1/2} \leq C_0 \epsilon, \ \ \   \forall\  |I|+|J| \leq N+1.
\end{equation}
Next we assume  the following bootstrap assumptions hold for all $s\in [s_0, s_1):$
\begin{equation} \label{bootassuptions}
\Ecal^D(s, \partial^I \widehat{L}^J \psi)^{1/2} \leq C_1 \epsilon, \quad  \ |I| + |J| \leq N+1,
\end{equation}
where $C_1> C_0$ is a constant to be specified later, and $\epsilon \ll 1$ measures the size of the initial data. The hyperbolic time $s_1$ is defined as 
\begin{equation*}
s_1:= \sup\{s: s>s_0, \ \eqref{bootassuptions}\ \textrm{holds} \}.
\end{equation*}
Using the bootstrap assumption and the energy bounds, we can easily get the following $L^2$ and $L^{\infty}$ estimates. 
\begin{proposition} \label{l2linftyestimate}
Suppose that the estimates in \eqref{bootassuptions} hold, then for all $s\in [s_0, s_1)$,  we have
\begin{align}
\big\|(s/t) \partial^I\widehat{L}^J \psi \big\|_{L_f^2(\Hcal_s)} + \big\|(\partial^I \widehat{L}^J \psi)_{-} \big\|_{L_f^2(\Hcal_s)}  & \lesssim C_1 \epsilon, \ \quad |I|+|J| \leq N+1,  \label{l2estimate} \\
\sup_{(t,x)\in \Hcal_s} \big(s \big|\partial^I \widehat{L}^J \psi \big| + t\big|(\partial^I \widehat{L}^J \psi)_{-} \big| \big) &\lesssim C_1 \epsilon, \ \quad  |I|+|J| \leq N-1. \label{linftyestimate}
\end{align}
\end{proposition}
\begin{proof}
For any function $\Phi$,  recall that $(\Phi)_{-} = (I_2 - (x_a/t) \gamma^0\gamma^a ) \Phi$,   the estimates in \eqref{l2estimate} then follows directly from the definition of the energy functional $\Ecal^D(s, \partial^I \widehat{L}^J \psi)$ and inequality \eqref{energybound}. 

As to the $L^{\infty}$ estimate, applying the commutator estimates in Lemma \ref{lem:commu} and Corollary \ref{cor:sobolev}, one can easily obtain 
\[
\sup_{(t,x)\in \Hcal_s} s \big|\partial^I\widehat{L}^J \psi \big| \lesssim C_1 \epsilon, \ \ \ \ |I|+|J| \leq N-1.
\]
On the other hand,  thanks to Corollary \ref{cor:sobolev}, we have 
\begin{equation} \label{decayestimate}
\begin{split}
&\sup_{(t,x)\in \Hcal_s} t \big|(\partial^I \widehat{L}^J \psi)_{-} \big| \\
&  \leq C \sum_{|K|\leq 2} \big \|\widehat{L}^K (I_2-(x_a/t) \gamma^0 \gamma^a) \partial^I \widehat{L}^J \psi \big\|_{L_f^2(\Hcal_s)}  \\
& \leq C   \big\| (\partial^I \widehat{L}^J \psi)_{-} \big\|_{L_f^2(\Hcal_s)}+  C\sum_{1\leq |K|\leq 2} \big\|(I_2-(x_a/t) \gamma^0 \gamma^a)  \widehat{L}^K \partial^I \widehat{L}^J \psi \big\|_{L_f^2(\Hcal_s)}  \\
&  + C \sum_{1\leq |K|\leq 2} \big\|\big[\widehat{L}^K , (I_2-(x_a/t) \gamma^0 \gamma^a)\big] \partial^I \widehat{L}^J \psi \big\|_{L_f^2(\Hcal_s)}.
\end{split}
\end{equation}
For any function $\Phi$,   by a direct calculation, one can find  
\begin{equation*}
\begin{split}
\big[\widehat{L}_b, I_2 - (x_a/t)\gamma^0 \gamma^a \big] \Phi= -(\gamma^0 \gamma^b + (x_b/t)) (\Phi)_{-}.
\end{split}
\end{equation*}
Additionally, 
\begin{equation*}
\begin{split}
[\widehat{L}_c \widehat{L}_b, I_2 - (x_a/t) \gamma^0 \gamma^a] \Phi &= -(\gamma^0 \gamma^c + (x_c/t)) (\widehat{L}_b \Phi)_{-} - (\gamma^0 \gamma^b + (x^b/t)) (\widehat{L}_c \Phi)_{-}  \\
& \quad + [(x_b/t) \gamma^0 \gamma^c + (x_c/t) \gamma^0 \gamma^b + 2(x_b x_c)/t^2] (\Phi)_{-}.
\end{split}
\end{equation*}
One can also refer to \cite{DoWy21} for the  calculations of the commutators.  Noticing that $|x|/t \leq 1$ in the cone $\Kcal$, inserting the above estimates into \eqref{decayestimate} and using Lemma  \ref{lem:commu}, one can show 
\[
\sup_{(t,x)\in \Hcal_s} t \big|(\partial^I \widehat{L}^J \psi)_{-} \big|  \leq C \sum_{|I_1|+|J_1| \leq N+1}  \big\|\big(\partial^{I_1} \widehat{L}^{J_1} \psi \big )_{-} \big\|_{L_f^2(\Hcal_s)} \lesssim C_1 \epsilon, \ \ \  |I|+|J| \leq N-1.
\]
This yields the desired result.
\end{proof}

\subsection{Refined estimates and proof of Theorem \ref{thm:mainresult}}

This part is devoted to obtaining better estimates for the Dirac field so as to close the bootstrap argument. First, we present two lemmas concerning the $L^2$ and $L^{\infty}$ estimates of the nonlinear term.

\begin{lemma} \label{lem:l2estimate}
Suppose the estimates in \eqref{bootassuptions} hold, then for $s\in [s_0, s_1)$, we have 
\begin{align*}
\big\|\partial^I {L}^J (\psi^*\gamma^0 \psi) \big\|_{L_f^2(\Hcal_s)} \lesssim (C_1 \epsilon)^2 s^{-1}, \quad  \ \ |I|+ |J| \leq N+1.
\end{align*}
\end{lemma}
\begin{proof}
According to Lemma \ref{modifiedlorentz},  we have 
\[
\big|\partial^I {L}^J(\psi^* \gamma^0 \psi)\big| \leq  \sum_{ \substack{|I_1|+|I_2| \leq |I|\\ |J_1|+|J_2| \leq |J|}} \big| (\partial^{I_1}\widehat{L}^{J_1}\psi)^*\gamma^0 (\partial^{I_2}\widehat{L}^{J_2} \psi)\big|.
\]
Thanks to Lemma \ref{lem:struc}, we can further get
\begin{equation} \label{kgstructure}
\begin{split}
&\big|\partial^I {L}^J(\psi^* \gamma^0 \psi)\big| 
\\
& \leq \sum_{ \substack{|I_1|+|I_2| \leq |I|\\ |J_1|+|J_2| \leq |J|}} \bigg( \big| (\partial^{I_1}\widehat{L}^{J_1}\psi)_{-}^*\gamma^0 (\partial^{I_2}\widehat{L}^{J_2} \psi)_{-}\big|  
  + \big| (\partial^{I_1}\widehat{L}^{J_1}\psi)_{-}^*\gamma^0 (\partial^{I_2}\widehat{L}^{J_2} \psi)_{+}\big|
\\
& \hskip2.1cm  + \big| (\partial^{I_1}\widehat{L}^{J_1}\psi)_{+}^*\gamma^0 (\partial^{I_2}\widehat{L}^{J_2} \psi)_{-}\big|  
 +\frac{s^2}{t^2} \big| (\partial^{I_1}\widehat{L}^{J_1}\psi)^*\gamma^0 (\partial^{I_2}\widehat{L}^{J_2} \psi)\big|   \bigg).
\end{split}
\end{equation}
Next we treat the above four terms in turn. Firstly, by Proposition \ref{l2linftyestimate}, it follows
\begin{equation*}
\begin{split}
& \sum_{\substack {|I_1| + |I_2| + |J_1| \\ +|J_2| \leq N+1}}
\big\| (\partial^{I_1}\widehat{L}^{J_1}\psi)_{-}^*\gamma^0 (\partial^{I_2}\widehat{L}^{J_2} \psi)_{-} \big\|_{L_f^2(\Hcal_s)} \\
& \lesssim \sum_{\substack{|I_1| + |J_1|\leq N-1 \\ |I_2|+|J_2| \leq N+1}}  \big\|(\partial^{I_1}\widehat{L}^{J_1}\psi)_{-}  \big\|_{L^{\infty}(\Hcal_s)} \big\|(\partial^{I_2}\widehat{L}^{J_2} \psi)_{-} \big\|_{L_f^2(\Hcal_s)} \\
& \lesssim (C_1 \epsilon)^2 s^{-1},
\end{split}
\end{equation*} 
where $N\geq 2,\, s\leq t$ is used in the first inequality.  Concerning the second term,  recall that $(\partial^I \widehat{L}^J \psi)_+$ shares the same estimate as $\partial^I \widehat{L}^J \psi$, we have 
\begin{equation*}
\begin{split}
& \sum_{\substack {|I_1| + |I_2| + |J_1| \\ +|J_2| \leq N+1}}
\big\| (\partial^{I_1}\widehat{L}^{J_1}\psi)_{-}^*\gamma^0 (\partial^{I_2}\widehat{L}^{J_2} \psi)_{+} \big\|_{L_f^2(\Hcal_s)} \\
& \lesssim \sum_{\substack{|I_1| + |J_1|\leq N-1 \\ |I_2|+|J_2| \leq N+1}}  \big\|(t/s)(\partial^{I_1}\widehat{L}^{J_1}\psi)_{-} \big\|_{L^{\infty}(\Hcal_s)} \big\|(s/t)(\partial^{I_2}\widehat{L}^{J_2} \psi)_{+} \big\|_{L_f^2(\Hcal_s)} \\ 
& + \sum_{\substack{|I_1| + |J_1|\leq N+1 \\ |I_2|+|J_2| \leq N-1}} \big \|(\partial^{I_1}\widehat{L}^{J_1}\psi)_{-} \big\|_{L_f^{2}(\Hcal_s)} \big\|(\partial^{I_2}\widehat{L}^{J_2} \psi)_{+} \big\|_{L^{\infty}(\Hcal_s)} \\
& \lesssim (C_1 \epsilon)^2 s^{-1}.
\end{split}
\end{equation*}
The estimate of the  third term  can be showed in a similar way as the second term, so we also have 
\[
\sum_{\substack {|I_1| + |I_2| + |J_1| \\ +|J_2| \leq N+1}}
\big\| (\partial^{I_1}\widehat{L}^{J_1}\psi)_{+}^*\gamma^0 (\partial^{I_2}\widehat{L}^{J_2} \psi)_{-} \big\|_{L_f^2(\Hcal_s)} \lesssim (C_1 \epsilon)^2 s^{-1}.
\]
Now let us bound the last term, it follows
\begin{equation*}
\begin{split}
&\sum_{\substack {|I_1| + |I_2| + |J_1| \\ +|J_2| \leq N+1}}
\big\|( s^2/t^2)(\partial^{I_1}\widehat{L}^{J_1}\psi)^*\gamma^0 (\partial^{I_2}\widehat{L}^{J_2} \psi)\big\|_{L_f^2(\Hcal_s)} \\
& \lesssim \sum_{\substack{|I_1| + |J_1|\leq N-1 \\ |I_2|+|J_2| \leq N+1}} \big \|(s/t)(\partial^{I_1}\widehat{L}^{J_1}\psi) \big\|_{L^{\infty}(\Hcal_s)} \big\|(s/t)(\partial^{I_2}\widehat{L}^{J_2} \psi) \big\|_{L_f^2(\Hcal_s)} \\
& \lesssim (C_1\epsilon)^2 s^{-1}.
\end{split}
\end{equation*}
Gathering the above four estimates, we get 
\[
\big\|\partial^I {L}^J (\psi^*\gamma^0 \psi) \big\|_{L_f^2(\Hcal_s)} \lesssim (C_1 \epsilon)^2 s^{-1}, \quad  \ \ |I|+ |J| \leq N+1.
\]
The proof is complete.
\end{proof}

\begin{lemma} \label{lem:linftyestimate}
Suppose the estimates in \eqref{bootassuptions} hold, then for $s\in [s_0, s_1)$, we have 
\begin{align*}
\big\|(t/s)\partial^{I} {L}^J (\psi^*\gamma^0 \psi)\big\|_{L^{\infty}(\Hcal_s)} \lesssim (C_1\epsilon)^2 s^{-2}, \ \quad |I|+|J| \leq N-1.
\end{align*}
\end{lemma}
\begin{proof}
The argument is similar to that of Lemma \ref{l2estimate}. Indeed, following \eqref{kgstructure}, we have 
\begin{equation*}
\begin{split}
&(t/s)\big|\partial^I {L}^J(\psi^* \gamma^0 \psi)\big| 
\\
& \leq  \frac{t}{s} \sum_{ \substack{|I_1|+|I_2| \leq |I|\\ |J_1|+|J_2| \leq |J|}} \bigg( \big| (\partial^{I_1}\widehat{L}^{J_1}\psi)_{-}^*\gamma^0 (\partial^{I_2}\widehat{L}^{J_2} \psi)_{-}\big| 
  + \big| (\partial^{I_1}\widehat{L}^{J_1}\psi)_{-}^*\gamma^0 (\partial^{I_2}\widehat{L}^{J_2} \psi)_{+}\big|
  \\
& \hskip2.5cm + \big| (\partial^{I_1}\widehat{L}^{J_1}\psi)_{+}^*\gamma^0 (\partial^{I_2}\widehat{L}^{J_2} \psi)_{-}\big|  
 +\frac{s^2}{t^2} \big| (\partial^{I_1}\widehat{L}^{J_1}\psi)^*\gamma^0 (\partial^{I_2}\widehat{L}^{J_2} \psi)\big|   \bigg).
\end{split}
\end{equation*}
Concerning the first term in the above bracket, one can see
\begin{equation*}
\begin{split}
& \sum_{\substack {|I_1| + |I_2| + |J_1| \\ +|J_2| \leq N-1}}
\big\|(t/s) (\partial^{I_1}\widehat{L}^{J_1}\psi)_{-}^*\gamma^0 (\partial^{I_2}\widehat{L}^{J_2} \psi)_{-} \big\|_{L^{\infty}(\Hcal_s)} \\
& \lesssim \sum_{\substack {|I_1| + |I_2| + |J_1| \\ +|J_2| \leq N-1}} \big\|(t/s)(\partial^{I_1}\widehat{L}^{J_1}\psi)_{-} \big\|_{L^{\infty}(\Hcal_s)} \big\|(\partial^{I_2}\widehat{L}^{J_2} \psi)_{-} \big\|_{L^{\infty}(\Hcal_s)} \\
& \lesssim (C_1 \epsilon)^2 s^{-2}.
\end{split}
\end{equation*} 
As to the second term, one can have 
\begin{equation*}
\begin{split}
& \sum_{\substack {|I_1| + |I_2| + |J_1| \\ +|J_2| \leq N-1}}
\big\|(t/s) (\partial^{I_1}\widehat{L}^{J_1}\psi)_{-}^*\gamma^0 (\partial^{I_2}\widehat{L}^{J_2} \psi)_{+} \big\|_{L^{\infty}(\Hcal_s)} \\
& \lesssim \sum_{\substack {|I_1| + |I_2| + |J_1| \\ +|J_2| \leq N-1}} \big\|(t/s)(\partial^{I_1}\widehat{L}^{J_1}\psi)_{-} \big\|_{L^{\infty}(\Hcal_s)} \big\|(\partial^{I_2}\widehat{L}^{J_2} \psi)_{+} \big\|_{L^{\infty}(\Hcal_s)} \\
& \lesssim (C_1 \epsilon)^2 s^{-2}.
\end{split}
\end{equation*} 
Similarly,
\begin{equation*}
\begin{split}
& \sum_{\substack {|I_1| + |I_2| + |J_1| \\ +|J_2| \leq N-1}}
\big\|(t/s)(\partial^{I_1}\widehat{L}^{J_1}\psi)_{+}^*\gamma^0 (\partial^{I_2}\widehat{L}^{J_2} \psi)_{-} \big\|_{L^{\infty}(\Hcal_s)} \\
& \lesssim \sum_{\substack {|I_1| + |I_2| + |J_1| \\ +|J_2| \leq N-1}} \big\|(\partial^{I_1}\widehat{L}^{J_1}\psi)_{+} \big\|_{L^{\infty}(\Hcal_s)} \big\|(t/s)(\partial^{I_2}\widehat{L}^{J_2} \psi)_{-} \big\|_{L^{\infty}(\Hcal_s)} \\
& \lesssim (C_1 \epsilon)^2 s^{-2}.
\end{split}
\end{equation*} 
Concerning the last term, one can see
\begin{equation*}
\begin{split}
& \sum_{\substack {|I_1| + |I_2| + |J_1| \\ +|J_2| \leq N-1}}
\big\|(t/s)(s^2/t^2)(\partial^{I_1}\widehat{L}^{J_1}\psi)^*\gamma^0 (\partial^{I_2}\widehat{L}^{J_2} \psi)\big\|_{L^{\infty}(\Hcal_s)} \\
& \lesssim \sum_{\substack {|I_1| + |I_2| + |J_1| \\ +|J_2| \leq N-1}} \big\|(s/t)(\partial^{I_1}\widehat{L}^{J_1}\psi) \big\|_{L^{\infty}(\Hcal_s)} \big\|(\partial^{I_2}\widehat{L}^{J_2} \psi) \big\|_{L^{\infty}(\Hcal_s)} \\
& \lesssim (C_1 \epsilon)^2 s^{-2}.
\end{split}
\end{equation*}
Combining the above estimates, the desired result then follows. 
\end{proof}

Now we can show the improved bounds for the Dirac field. 
\begin{proposition} \label{improvebounds}
Let the estimates in \eqref{bootassuptions} hold, then there exists some constant $C$ depending on $N$ only, such that for all $s\in [s_0, s_1)$, we have 
\begin{align*}
\Ecal^D(s, \partial^I \widehat{L}^J \psi)^{1/2} \leq C_0\epsilon + C(C_1 \epsilon)^3, \ \ \quad  |I|+|J| \leq N+1.
\end{align*}
\end{proposition}
\begin{proof}
Since $\psi$ solves the Dirac equation \eqref{eq:model-Dirac}, one can easily see
\[
i\gamma^{\mu} \partial_{\mu} (\partial^I \widehat{L}^J \psi) + m (\partial^I \widehat{L}^J \psi) =  \partial^I \widehat{L}^J  \big[(\psi^*\gamma^0 \psi) \psi  \big],
\]
where we used the fact $[i\gamma^{\mu} \partial_{\mu}, \partial^I \widehat{L}^J] = 0$. Following Proposition \ref{boundsofenergy}, we have
\begin{equation*}
\begin{split}
\Ecal^D(s,\partial^I \widehat{L}^J \psi)^{1/2} \leq \Ecal^D(s_0, \partial^I \widehat{L}^J \psi )^{1/2} + \int_{s_0}^s \big\|\partial^I \widehat{L}^J  \big[(\psi^*\gamma^0 \psi) \psi  \big]  \big\|_{L_f^2(\Hcal_{\tau})} {\rm d}\tau\,.
\end{split}
\end{equation*}
By a straightforward calculation, one can get
\[
\widehat{L}_a  \big[(\psi^*\gamma^0 \psi) \psi  \big] = \big(L_a(\psi^*\gamma^0\psi) \big) \psi + (\psi^*\gamma^0\psi)\widehat{L}_a \psi.
\]
Using the Leibniz rule, we can see 
\begin{align*}
\partial^I \widehat{L}^J  \big[(\psi^*\gamma^0 \psi) \psi  \big] = \sum_{\substack{|I_1|+|I_2| = I \\ |J_1|+|J_2|= J}}
\big[ \partial^{I_1}L^{J_1} (\psi^*\gamma^0 \psi) \big] \partial^{I_2} \widehat{L}^{J_2} \psi.
\end{align*}
In view of Lemma \ref{lem:l2estimate} and Lemma \ref{lem:linftyestimate}, we conclude 
\begin{equation*}
\begin{split}
& \big\|\partial^I \widehat{L}^J  \big[(\psi^*\gamma^0 \psi) \psi  \big]  \big\|_{L_f^2(\Hcal_{\tau})} \\
&\leq \sum_{\substack{|I_1|+|I_2|+|J_1| \\+|J_2| \leq N+1}} \big\|\big[ \partial^{I_1}L^{J_1} (\psi^*\gamma^0 \psi) \big] \partial^{I_2} \widehat{L}^{J_2} \psi  \big\|_{L_f^2(\Hcal_{\tau})} \\
&\leq \sum_{\substack{|I_2|+|J_2|\leq N-1 \\ |I_1|+|J_1|\leq N+1}} \big\|\big[ \partial^{I_1}L^{J_1} (\psi^*\gamma^0 \psi) \big] \big\|_{L_f^2(\Hcal_{\tau})} 
\big\|\partial^{I_2} \widehat{L}^{J_2} \psi  \big\|_{L^{\infty}(\Hcal_{\tau})} \\
& + \sum_{\substack{|I_2|+|J_2|\leq N+1 \\ |I_1|+|J_1|\leq N-1}} \big\|(t/\tau)\big[ \partial^{I_1}L^{J_1} (\psi^*\gamma^0 \psi) \big] \big\|_{L^{\infty}(\Hcal_{\tau})} 
\big\|(\tau/t)\partial^{I_2} \widehat{L}^{J_2} \psi  \big\|_{L_f^{2}(\Hcal_{\tau})} \\
& \lesssim (C_1 \epsilon)^3 \tau^{-2},
\end{split}
\end{equation*}
where $N\geq 2$ is used in the second inequality. Noting that $s_0=2$ and inequality \eqref{initialcondition}, one can finally see
\begin{align*}
\Ecal^D(s,\partial^I \widehat{L}^J \psi)^{1/2} & \leq C_0\epsilon + \int_{2}^s C(C_1\epsilon)^3 \tau^{-2}  {\rm d}\tau \\
& \leq C_0\epsilon + C(C_1\epsilon)^3.
\end{align*}
The proof is done.
\end{proof}

Now we are ready to prove Theorem \ref{thm:mainresult}. 
\begin{proof}[Proof of Theorem \ref{thm:mainresult}]
For fixed $N\geq 2$, we can choose $C_1>4C_0$, then set $\epsilon_0$ so small that 
$C(C_1\epsilon_0)^2 \leq 1/4$.  For all $0<\epsilon\leq \epsilon_0$ and  $s\in [s_0,s_1)$, we have from Proposition \ref{improvebounds}
\begin{align} \label{betterbounds}
\Ecal^D(s, \partial^I \widehat{L}^J \psi)^{1/2} \leq \frac{1}{2} C_1 \epsilon_0, \ \ \ |I|+|J| \leq N+1.
\end{align} 
Now we  claim $s_1 =+\infty$, otherwise, i.e. $s_1 <+\infty$, then one of  the inequalities (for some $I, J$) in \eqref{bootassuptions} should be an equality. However, the inequalities in \eqref{betterbounds} implies the corresponding bound in \eqref{bootassuptions} can be refined, and this is a contradiction. Hence, the Dirac equation  \eqref{eq:model-Dirac}  asscociated with small initial data $\psi_0$ admits a global solution $\psi$, and the estimates in  \eqref{bootassuptions} hold for 
all $s\in [s_0, \infty)$ and $\epsilon\leq \epsilon_0$, from which we can deduce the following pointwise estimates (see Proposition \ref{l2linftyestimate})
\begin{align} \label{diraclinftyesti}
\sup_{\Hcal_s}s|\psi(t,x)| \lesssim C_1\epsilon.
\end{align}
Noting that in $\Hcal_s \cap \Kcal$, we have $s\sim t^{1/2}(t-r)^{1/2}$. Thus \eqref{diraclinftyesti} yields 
\begin{align} \label{wavedecay}
|\psi(t,x)| \lesssim \frac{C_1 \epsilon}{t^{1/2}(t-r)^{1/2}}.
\end{align}
 
Next, let us prove the unified pointwise decay with respect to $m \in [0,1]$. Using the semi-hyperboloidal frame, we can rewrite the Dirac equation as (see also \cite{DoLeWy21})
\begin{align*}
i\Big(\gamma^0 - \frac{x_a}{t} \gamma^a \Big)\partial_t \psi + i \gamma^a \underline{\partial}_a \psi + m\psi = (\psi^*\gamma^0\psi)\psi.
\end{align*}
Recall that $(\psi)_{-} = (I_2- (x_a/t)\gamma^0\gamma^a)\psi$ and $\underline{\partial}_a = L_a/t$, one can obtain 
\begin{equation}
\begin{split} \label{kgdecay}
mt|\psi(t,x)| & \leq t\Big|\big(\gamma^0 - \frac{x_a}{t}\gamma^a \big) \partial_t \psi\Big| + t|\gamma^a \underline{\partial}_a \psi| + t |(\psi^*\gamma^0\psi)\psi| \\
& \lesssim t \Big|\big( \gamma^0(I_2 - \frac{x_a}{t}\gamma^0\gamma^a \big) \partial_t \psi\Big| + |L_a \psi| + t|\psi^*\gamma^0\psi| |\psi| \\
& \lesssim t|(\partial \psi)_{-}| + |L_a \psi| + t|\psi^*\gamma^0\psi| |\psi| \\
& \lesssim C_1 \epsilon, 
\end{split}
\end{equation}
where Proposition \ref{l2linftyestimate} and Lemma \ref{lem:linftyestimate} are used in the last inequality.  Combining the estimates \eqref{wavedecay} and \eqref{kgdecay}, we finally obtain
\[
|\psi(t,x)| \lesssim \frac{C_1 \epsilon}{t^{1/2}(t-r)^{1/2}+mt }.
\]
The proof is finished.
\end{proof}

\section{Scattering for the Dirac field} \label{ch:scattering}
In this section, we study the long time behavior of the global solution to the Dirac equation in the Sobolev space. The main ingredient is the pointwise decay result obtained in the previous section and the ghost weight energy estimate due to  Alinhac. 

\subsection{Ghost weight energy estimate}

\begin{proposition} \label{prop:ghostenergybound}
	Let $\psi$ be the solution to the Dirac equation \eqref{eq:model-Dirac}, then for any $t\geq t_0$, we have the following the ghost weight energy estimate:
	\begin{align}
	\|\psi(t)\|_{L_x^2} + \bigg(\int_{t_0}^t \Big\|\frac{\psi-\frac{x_j}{r}\gamma^0\gamma^j\psi}{\langle \tau -r \rangle}  \Big\|_{L_x^2}^2\, \rd\tau  \bigg)^{1/2} \leq C   \|\psi(t_0)\|_{L_x^2} + C\int_{t_0}^t \|F\|_{L_x^2}\, \rd\tau, \nonumber
	\end{align}
where $C$ is an absolute constant and $r=|x|$.
\end{proposition}

\begin{proof}
	Let $q(x,t)= \arctan (r-t)$, multiplying $-i e^q \psi^* \gamma^0$ to both sides of \eqref{eq:model-Dirac}, and we find 
	\begin{align} \label{eq:weight}
		e^q \psi^* \partial_t \psi + e^q \psi^*\gamma^0\gamma^j\partial_j \psi - im e^q \psi^*\gamma^0\psi = -i e^q \psi^*\gamma^0 F(\psi).
	\end{align}
	Taking the complex conjuagte of \eqref{eq:weight}, one can see 
	\begin{align}
		e^q \partial_t \psi^* \psi + e^q \partial_j \psi^* \gamma^0\gamma^j \psi + ime^q \psi^*\gamma^0\psi= i e^q F(\psi)^*\gamma^0\psi.
	\end{align}
	Summing the above two equalities and using Leibniz rule, we get
	\begin{align} \label{ghostequality}
		\partial_t (e^q \psi^*\psi) + \partial_j (e^q \psi^*\gamma^0\gamma^j\psi) - (\partial_t q) e^q \psi^*\psi& -(\partial_j q)e^q \psi^*\gamma^0\gamma^j \psi  \nonumber \\ 
		&= ie^q[F(\psi)^*\gamma^0\psi-\psi^*\gamma^0F(\psi)].
	\end{align} 
	By a direct computation, we have 
	\begin{align*}
		- (\partial_t q) e^q \psi^*\psi& -(\partial_j q)e^q \psi^*\gamma^0\gamma^j \psi = \frac{1}{1+|r-t|^2} e^q \big[\psi^*\psi -\frac{x_j}{r}\psi^*\gamma^0\gamma^j\psi \big].
	\end{align*}
	On the other hand, we find 
	\begin{align*}
	& \Big(\psi - \frac{x_j}{r}\gamma^0\gamma^j \psi \Big)^*\Big(\psi - \frac{x_k}{r}\gamma^0\gamma^k \psi \Big) \\
	&= \psi^*\psi - 2\frac{x_j}{r}\gamma^0\gamma^j\psi -\frac{x_jx_k}{r^2}\psi^*\gamma^j\gamma^k \psi  \\
	&= 2 \Big( \psi^*\psi - \frac{x_j}{r}\gamma^0\gamma^j \psi\Big).
	\end{align*}
	This immediately yields 
	\[
	-(\partial_t q) e^q \psi^*\psi -(\partial_j q)e^q \psi^*\gamma^0\gamma^j \psi = \frac{1}{2(1+|r-t|^2)} e^q \Big|\psi-\frac{x_j}{r}\gamma^0\gamma^j \psi\Big|^2.
	\]
	Integrating \eqref{ghostequality} over the domain $\mathbb{R}^2 $, we can reach 
	\begin{align}
	\partial_t  \|e^{q/2}\psi\|^2_{L_x^2} +  \frac 12 \Big\| e^{q/2}\frac{\big|\psi- (x_j/r)\gamma^0\gamma^j \psi\big|}{\langle t-r\rangle} \Big\|_{L_x^2}^2 &= i\int e^q\big[F(\psi)^*\gamma^0\psi - \psi^*\gamma^0F(\psi)\,\rd x \big]  \nonumber \\
	& \leq C\|F(\psi)\|_{L_x^2} \|e^{q/2}\psi\|_{L_x^2}, \nonumber
 	\end{align}
	where we use the fact $e^{q/2} \leq e^{\pi/4}$. Thus  we  further obtain 
	\begin{align*}
	&\|e^{q/2}\psi(t)\|_{L_x^2} + \Big(\int_{t_0}^t \Big\| e^{q/2}\frac{\big|\psi- (x_j/r)\gamma^0\gamma^j \psi\big|}{\langle \tau-r\rangle} \Big\|_{L_x^2}^2\,\rd\tau \Big)^{1/2}  \\
	& \leq C\Big( \|e^{q/2} \psi(t_0)\|_{L_x^2}  + \int_{t_0}^t \|F(\psi)(\tau)\|_{L_x^2}\, \rd\tau  \Big).
	\end{align*}
	The desired result then follows, as $e^{q/2} \sim 1$.
\end{proof}

To proceed, we introduce some notations. For any $\mathbb{C}^2$-valued function $\Phi$, let 
\[
[\Phi]_{+} := \Phi + \frac{x_j}{r}\gamma^0\gamma^j \Phi , \qquad \ [\Phi]_{-} := \Phi-\frac{x_j}{r}\gamma^0\gamma^j \Phi.
\]
Note that there is a slight difference between the notations $[\Phi]_-, [\Phi]_+$ and $(\Phi)_{-}, (\Phi)_{+}$, and the latter was defined in Section \ref{sec:hiddenstruc}. Now we present a result which is similar to Lemma \ref{lem:struc}.
\begin{lemma} \label{lem:ghostdecomp}
	Let $\Psi, \Phi$ be two $\mathbb{C}^2$-valued functions, then 
\begin{align*}
\Psi^*\gamma^0\Phi = \frac{1}{4} \Big([\Psi]_{-}^* \gamma^0 [\Phi]_{-} + [\Psi]_{-}^* \gamma^0 [\Phi]_+ + [\Psi]_+^* \gamma^0 [\Phi]_{-}   \Big).
\end{align*}
\end{lemma}
\begin{proof}
	It is easy to see 
	\begin{align*}
	\Psi^*\gamma^0\Phi & = \frac{1}{4} ([\Psi]_{+}^* + [\Psi]_{-}^*)\gamma^0 ([\Phi]_+ + [\Phi]_{-}) \\
	&=\frac{1}{4}\Big([\Psi]_{-}^* \gamma^0 [\Phi]_{-} + [\Psi]_{-}^* \gamma^0 [\Phi]_+ + [\Psi]_+^* \gamma^0 [\Phi]_{-} + [\Psi]_+^* \gamma^0 [\Phi]_{+}   \Big).
	\end{align*}
	By a straightforward computation, we have 
	\begin{align*}
		[\Psi]_+^* \gamma^0 [\Phi]_{+} &= \Psi^*\gamma^0\Phi + \frac{x_k}{r}\Psi^*\gamma^k \Phi -\frac{x_j}{r}\Psi^*\gamma^j \Phi + \frac{x_jx_k}{r^2}\Psi^*\gamma^0\gamma^j\gamma^k\Phi \\
		&=\Psi^*\gamma^0\Phi + \frac{x_j^2}{r^2} \Psi^*\gamma^0(\gamma^j)^2\Phi + \sum_{j<k} \frac{x_jx_k}{r^2}\Psi^*\gamma_0(\gamma^j\gamma^k+\gamma^k\gamma^j)\Phi \\
		&=0.
	\end{align*}
	This yields the desired result.
\end{proof}
We will perform the ghost weight energy estimate for the global Dirac solution obtained in Theorem \ref{eq:model-Dirac}. First we collect  some key  estimates which will be frequently used, according to Lemma \ref{lem:linftyestimate},
   \begin{align} \label{nullterm1}
   	|\partial^{I}L^J(\psi^*\gamma^0\psi)(t,x)|\lesssim (C_1\epsilon)^2 t^{-3/2}(t-r)^{-1/2}, \ \ \  \ |I|+|J|\leq N-1.
   \end{align}
 In addition, for $|I|+|J| \leq N-1$,   Proposition \ref{l2linftyestimate} infers 
 \begin{align}
 	|(\partial^IL^J \psi)_{-}(t,x)|+ |(\partial^I\widehat{L}^J \psi)_{-}(t,x)| & \lesssim C_1\epsilon t^{-1},  \nonumber \\
 	|\partial^IL^J \psi(t,x)| + |\partial^I\widehat{L}^J \psi(t,x)| &\lesssim C_1\epsilon t^{-1/2}(t-r)^{-1/2}.  \label{eq:linftydecaysolution}
 \end{align}
 Noting that 
 \begin{align*}
 	\big|[\partial^IL^J \psi]_{-} - (\partial^IL^J \psi)_{-} \big|\leq \Big| \frac{x_j}{r} \frac{t-r}{t}\gamma^0\gamma^j \partial^IL^J\psi \Big| \lesssim C_1 \epsilon t^{-1}, \ \ \ |I|+|J| \leq N-1,
 \end{align*}
 where we  used \eqref{eq:linftydecaysolution}. Hence 
 \begin{align}
 	|[\partial^IL^J\psi]_{-}| \lesssim C_1\epsilon t^{-1}, \ \ \  \  |I|+|J| \leq N-1. 
 \end{align}
  	
\begin{proposition} \label{prop:ghostenergy}
	Suppose the assumptions in Theorem \ref{thm:mainresult} hold. There exist some small constant $\epsilon_1\leq \epsilon_0$ and constant $C$, such that if $\|\psi_0\|_{H^{N+1}} <\epsilon \leq \epsilon_1$,  then the global solution $\psi$ associated with $\psi_0$ satisfies
	\begin{align}
		\sup_{t\geq t_0}\Ecal(t)^{1/2} \leq C \|\psi_0\|_{H^{N+1}},
    \end{align}
where 
\begin{align}
	\Ecal(t):= \sum_{|I|+|J|\leq N+1} \bigg(\sup_{t_0\leq \tau\leq t} \|\partial^IL^J \psi(\tau)\|^2 + \int_{t_0}^t \Big\|\frac{[\partial^IL^J\psi]_{-}}{\langle \tau-r \rangle} \Big\|_{L_x^2}^2 \rd\tau \bigg).
\end{align}
\end{proposition}

\begin{proof}
	Observing that  $\partial^I\widehat{L}^J \psi $ solves the following equation:
	\[
	i\gamma^{\mu}\partial_{\mu} \partial^I \widehat{L}^J \psi + m\, \partial^I \widehat{L}^J \psi = \partial^I \widehat{L}^J [(\psi^*\gamma^0\psi)\psi].
	\] 
	For any $|I|+|J| \leq N+1$ and $t\geq t_0$,  by Proposition \ref{prop:ghostenergybound}, we have 
	\begin{align} \label{eq:l1boundforghost}
		&\sup_{t_0\leq \tau \leq t}\|\partial^I\widehat{L}^J \psi(\tau)\|_{L_x^2} + \Big(\int_{t_0}^t \Big\|\frac{[\partial^I\widehat{L}^J\psi]_{-}}{\langle \tau-r \rangle} \Big\|_{L_x^2}^2 \,\rd\tau\Big)^{1/2} \nonumber \\
		& \leq C \Big(\|\partial^I\widehat{L}^J \psi_0\|_{L^2} + \int_{t_0}^t \big\|\partial^I \widehat{L}^J [(\psi^*\gamma^0\psi)\psi](\tau) \big\|_{L^2}\,\rd\tau \Big).
	\end{align}
	On the other hand, 
	\begin{align*}
\partial^I \widehat{L}^J  \big[(\psi^*\gamma^0 \psi) \psi  \big] &= \sum_{\substack{|I_1|+|I_2| = I \\ |J_1|+|J_2|= J}}
\big[ \partial^{I_1}L^{J_1} (\psi^*\gamma^0 \psi) \big] \partial^{I_2} \widehat{L}^{J_2} \psi.
\end{align*}
Then we find 
\begin{align*}
 \big\|\partial^I \widehat{L}^J [(\psi^*\gamma^0\psi)\psi] \big\|_{L^2} 
&\leq \sum_{\substack{|I_1|+|J_1| \leq  N-1 \\ |I_2|+|J_2|\leq N+1}}
\big\|\big[ \partial^{I_1}L^{J_1} (\psi^*\gamma^0 \psi) \big] \partial^{I_2} \widehat{L}^{J_2} \psi\big\|_{L_x^2} \\
& + \sum_{\substack{|I_1|+|J_1| \leq  N+1 \\ |I_2|+|J_2|\leq N-1}}
\big\|\big[ \partial^{I_1}L^{J_1} (\psi^*\gamma^0 \psi) \big] \partial^{I_2} \widehat{L}^{J_2} \psi\big\|_{L_x^2} \\
&= I + II.
\end{align*}
Following \eqref{nullterm1}, we can see
\begin{align*}
	I &\leq \sum_{\substack{|I_1|+|J_1|\leq  N-1 \\ |I_2|+|J_2|\leq N+1}}
\big\|\big[ \partial^{I_1}L^{J_1} (\psi^*\gamma^0 \psi) \big] \big\|_{L_x^{\infty}} \big\|\partial^{I_2} \widehat{L}^{J_2} \psi \big\|_{L_x^2} \\
&\lesssim (C_1 \epsilon)^2\tau^{-3/2} \sum_{|I_2|+|J_2|\leq N+1}\big\|\partial^{I_2} \widehat{L}^{J_2} \psi (\tau)\big\|_{L_x^2}.
\end{align*}
Owing to Leibniz rule and Lemma \ref{lem:ghostdecomp}, we have  
\begin{align*}
	II\leq &\sum_{\substack{|I_1|+|J_1| \leq  N+1 \\ |I_2|+|J_2|\leq N-1}} \sum_{\substack{|I_3|+|I_4|=|I_1|\\ |J_3|+|J_4|=|J_1|}}
\big\|\big[ (\partial^{I_3}L^{J_3} \psi)^*\gamma^0 \partial^{I_4}L^{J_4} \psi \big] \partial^{I_2} \widehat{L}^{J_2} \psi\big\|_{L_x^2} \\
& \lesssim  \sum_{\substack{|I_3|+|J_3| \leq  N+1 \\ |I_2|+|J_2|+|I_4|+|J_4|\leq N-1}} \Big( \big\|[\partial^{I_3}L^{J_3} \psi]_{-}[\partial^{I_4}L^{J_4} \psi]_{-}\partial^{I_2} \widehat{L}^{J_2} \psi \big\|_{L_x^2} \\
& \qquad \qquad \qquad  + \big\|[\partial^{I_3}L^{J_3} \psi]_{-}[\partial^{I_4}L^{J_4} \psi]_{+}\partial^{I_2} \widehat{L}^{J_2} \psi \big\|_{L_x^2} \\
& \qquad  \qquad  \qquad \qquad \qquad+ \big\|[\partial^{I_3}L^{J_3} \psi]_{+}[\partial^{I_4}L^{J_4} \psi]_{-}\partial^{I_2} \widehat{L}^{J_2} \psi \big\|_{L_x^2} \Big) \\
&\lesssim \sum_{|I_3|+|J_3|\leq N+1}(C_1\epsilon)^2 \Big(\tau^{-1}\Big\|\frac{[\partial^{I_3}L^{J_3} \psi]_{-}}{\langle\tau-r\rangle} \Big\|_{L_x^2}+ \tau^{-3/2} \big\|\partial^{I_3}L^{J_3} \psi \big\|_{L_x^2} \Big).
\end{align*}
The estimates of $I$ and $II$ lead to 
\begin{align*}
	 &\big\|\partial^I \widehat{L}^J [(\psi^*\gamma^0\psi)\psi] \big\|_{L^2} 
	 \\&\lesssim \sum_{|I_3|+|J_3|\leq N+1}(C_1\epsilon)^2 \Big(\tau^{-1}\Big\|\frac{[\partial^{I_3}L^{J_3} \psi]_{-}}{\langle\tau-r\rangle} \Big\|_{L_x^2}+ \tau^{-3/2} \big\|\partial^{I_3}L^{J_3} \psi \big\|_{L_x^2} \Big).
\end{align*}
In view of \eqref{eq:l1boundforghost}, we obtain
\begin{align*}
	&\sup_{t_0\leq \tau \leq t} \big\|\partial^I\widehat{L}^J \psi(\tau) \big\|_{L_x^2} + \Big(\int_{t_0}^t \Big\|\frac{[\partial^I\widehat{L}^J\psi]_{-}}{\langle \tau-r \rangle} \Big\|_{L_x^2}^2 \rd\tau\Big)^{1/2} \\
	&\leq C\big\|\partial^I\widehat{L}^J \psi_0 \big\|_{L^2} + C (C_1\epsilon)^2 \sum_{|I_3|+|J_3|\leq N+1} \bigg(\sup_{t_0\leq \tau \leq t} \big\|\partial^{I_3}L^{J_3} \psi(\tau) \big\|_{L_x^2} \\
	&  \qquad \qquad  + \Big(\int_{t_0}^t \Big\|\frac{[\partial^{I_3}L^{J_3} \psi]_{-}}{\langle\tau-r\rangle} \Big\|^2_{L_x^2} \rd\tau\Big)^{1/2} \bigg).
\end{align*}
This immediately gives rise to the following inequality
\begin{align*}
	\Ecal(t)^{1/2} \leq C \sum_{|I|+|J|\leq N+1}\|\partial^I\widehat{L}^J \psi_0\|_{L^2} + C (C_1\epsilon)^2 \Ecal(t)^{1/2}. 
\end{align*}
Now we take $\epsilon_1\leq \epsilon_0$ so small that $C(C_1\epsilon_1)^2 \leq 1/2$. Thus if $\epsilon<\epsilon_1$, we have   
\begin{align*}
	\Ecal(t)^{1/2} \leq C \sum_{|I|+|J|\leq N+1} \|\partial^I\widehat{L}^J \psi_0\|_{L^2}.
\end{align*}
Given that $\psi_0$ has compact support in space, the desired result follows. 
\end{proof}

\subsection{Long time asymptotic behavior}
In this part, we show the small global Dirac solution scatters as time tends to infinity via the ghost weight energy estimate. Let us first derive the integral formula for a general nonlinear Dirac equation. Suppose $\psi$ solves the following Dirac equation:
\begin{align} \label{NLD}
	i\gamma^{\mu} \partial_{\mu} \psi + m \psi = G, \ \ \  \psi(t_0)=\psi_0,
\end{align}
where $m\in [0,1], t_0=2$, $G$ is the external force and $\psi_0$ is the initial data. Taking Fourier transform with respect to space variable on both sides of \eqref{NLD}, we get
\begin{align}
\partial_t \mathcal{F}{\psi} +i(\xi_a\gamma^0\gamma^a-m\gamma^0)\mathcal{F}{\psi}=-i\gamma^0 \mathcal{F}{G}.
\end{align}
By a routine calculation, one can have
\begin{align} \label{duhamelfor}
	\psi(t,x) = S(t-t_0) \psi(t_0) -i \int_{t_0}^t S(t-\tau)\gamma^0 G\, \rd\tau,
\end{align}
where $S(t):= e^{it(i\gamma^0\gamma^a \partial_a + m\gamma^0)}$ is the matrix group propagator,  and for any vector valued  function $f\in \mathbb{C}^2$, 
\begin{align} \label{groupoperator}
S(t)f = \mathcal{F}^{-1} e^{it(-\gamma^0\gamma^a\xi_a + m\gamma^0)} \mathcal{F}f(\xi).
\end{align}
In \eqref{groupoperator}, we recall that, for any matrix $A$, $e^{A}$ is defined as 
\[
e^{A} = \sum_{n=0}^{\infty} \frac{A^n}{n!}.
\]
Since $\gamma^0\gamma^a$ and $\gamma^0$ are Hermitian matrices, we can claim $B:=i\gamma^0\gamma^a\partial_a+m\gamma^0$ is a self-adjoint operator on complex-valued Hilbert space $L^2\times L^2$, with $D(B)= H^1\times H^1$. This immediately infers that $S(t)=e^{itB}$ forms a unitary group on $L^2\times L^2$, and we shall mainly use the following properties of $S(t)$:
\begin{itemize}
\item[(1)] $S(0)= I_2$, $S(t)S(\tau) = S(t+\tau)$;
\item[(2)] $S(t)$ commutes with the operators $\partial_1,\partial_2$, and $\|S(t)f\|_{L_x^2} = \|f\|_{L_x^2}$.
\end{itemize}
The following theorem asserts that if we can bound the external force in appropriate space, then the nonlinear Dirac solution behaves like a linear Dirac solution as $t\to +\infty$. One can refer to \cite{Katayama17} for a similar result in the context of wave equation.

\begin{theorem} \label{thm:diracscatter}
	Let $\psi_0 \in H^s(\mathbb{R}^2)$, $s\in \mathbb{N}$. Suppose $\psi$ is a global solution to Dirac equation \eqref{NLD} and 
	\begin{align*}
		\int_{t_0}^{+\infty} \|G(\tau)\|_{H^s}\, \rd\tau <+\infty.
	\end{align*}
	Then there exist some function $\psi^+ \in H^s(\mathbb{R}^2)$ and constant $C$, such that 
	\begin{align*}
       \|\psi(t)- S(t-t_0)\psi^+\|_{H^s} \leq C \int_{t}^{+\infty} \|G(\tau)\|_{H^s}\, \rd\tau.
	\end{align*}
	In particular, 
	\begin{align*}
		\lim_{t\to +\infty} \|\psi(t)- S(t-t_0)\psi^+\|_{H^s} = 0.
	\end{align*}
\end{theorem}
\begin{proof}
	We set 
	\begin{align*}
		\psi^+= \psi(t_0) - i \int_{t_0}^{+\infty} S(t_0-\tau)\gamma^0G\,\rd\tau.
	\end{align*}
	This formula makes sense, since 
	\begin{align*}
		\|\psi^+\|_{H^s} \leq \|\psi(t_0)\|_{H^s} + C\int_{t_0}^{+\infty}\|G(\tau)\|_{H^s} \,\rd \tau < +\infty,
	\end{align*}
	where we have used $\|\partial_x^{\alpha}S(t) f\|_{L_x^2} = \|\partial^{\alpha}_x f\|_{L_x^2}, \ \forall\ \alpha = (\alpha_1,\alpha_2),\, \partial_x^{\alpha}:=\partial_{x_1}^{\alpha_1} \partial_{x_2}^{\alpha_2},\, |\alpha|\leq s$. Owing to \eqref{duhamelfor} and group property of $S(t)$, we can reach 
	\begin{align*}
		\|\psi(t)- S(t-t_0)\psi^+\|_{H^s} &\leq \Big\|\int_{t}^{+\infty}S(t-\tau)\gamma^0G\,\rd \tau \Big\|_{H^s} \\
		 & \leq C \int_{t}^{+\infty}\|G(\tau)\|_{H^s}\, \rd\tau.
	\end{align*}
	The proof is completed.
\end{proof}
\begin{remark}
	We point out that $S(t-t_0)\psi^+$ solves the linear Dirac equation $i\gamma^{\mu}\partial_{\mu} \psi + m \psi =0$ with initial data $\psi^+$ at $t=t_0$.
\end{remark}

\begin{proof}[Proof of Theorem \ref{thm:scatterresult}] 
According to Proposition \ref{prop:ghostenergy}, we have 
\begin{align}
	\sup_{t\geq t_0}\Ecal(t)^{1/2} \leq C \|\psi_0\|_{H^{N+1}} \leq C\epsilon.
\end{align}
In order to show \eqref{eq:scatterhighorder}, by Theorem \ref{thm:diracscatter}, it suffices to bound 
$$\int_{t_0}^{+\infty} \|(\psi^*\gamma^0\psi)\psi(\tau)\|_{H^{N+1}} \,\rd \tau.$$
Indeed,  
\begin{align*}
	&\sum_{|I|\leq N+1}\|\partial^{I}[(\psi^*\gamma^0\psi)\psi]\|_{L_x^2} \\
	& \leq \sum_{\substack{|I_1| \leq N-1\\|I_2|\leq N+1}} \|[\partial^{I_1}(\psi^*\gamma^0 \psi) ]\partial^{I_2}\psi\|_{L_x^2} + \sum_{\substack{|I_1|\leq N+1\\|I_2|\leq N-1}} \|[\partial^{I_1}(\psi^*\gamma^0 \psi) ]\partial^{I_2}\psi\|_{L_x^2} \\
	&= I+II.
\end{align*}
Utilizing  \eqref{nullterm1}, we can derive 
\begin{align} \label{eq:higherorderdec}
	|I| \lesssim (C_1\epsilon)^2 \tau^{-3/2} \sup_{\tau\geq t_0} \Ecal(\tau)^{1/2}.
\end{align}
On the estimate of $II$, once again, we use Lemma \ref{lem:ghostdecomp},
\begin{align} \label{eq:ghostdecay}
	|II| &\leq \sum_{\substack{|I_3|+|I_4|\leq N+1\\|I_2|\leq N-1}} \|[(\partial^{I_3}\psi)^*\gamma^0 \partial^{I_4}\psi]\partial^{I_2}\psi\|_{L_x^2} \nonumber\\
	& \lesssim   \sum_{\substack{|I_3|\leq N+1\\ |I_4|+|I_2|\leq N-1}} \Big(\|[\partial^{I_3}\psi]_{-}[\partial^{I_4}\psi]_{-} \partial^{I_2}\psi\|_{L_x^2} + \|[\partial^{I_3}\psi]_{-}[\partial^{I_4}\psi]_{+} \partial^{I_2}\psi\|_{L_x^2} \nonumber\\
	& \qquad \qquad \qquad + \|[\partial^{I_3}\psi]_{+}[\partial^{I_4}\psi]_{-} \partial^{I_2}\psi\|_{L_x^2}  \Big) \nonumber\\
	&\lesssim (C_1\epsilon)^2 \sum_{|I_3|\leq N+1}\Big( \tau^{-1} \Big\|\frac{[\partial^{I_3}\psi]_{-}}{\langle \tau-r\rangle} \Big\|_{L_x^2} + \tau^{-3/2} \|\partial^{I_3} \psi\|_{L_x^2} \Big).
\end{align}
The estimates of $I$ and $II$ lead to 
\begin{align}
	&\int_{t_0}^{+\infty} \|(\psi^*\gamma^0\psi)\psi(\tau)\|_{H^{N+1}}\,\rd \tau \nonumber\\
	&\leq C (C_1\epsilon)^2 \sup_{\tau\geq t_0} \Ecal(\tau)^{1/2}.
	\end{align}
Consequently, we conclude from Theorem \ref{thm:diracscatter} that there exists some $\psi^+ \in H^{N+1}$, such that 
\begin{align*}
&\|\psi(t) - S(t-t_0) \psi^+\|_{H^{N+1}} \\
&\leq C \int_{t}^{\infty} \|(\psi^*\gamma^0\psi)\psi(\tau)\|_{H^{N+1}}\,\rd \tau   \\
& \leq C(C_1\epsilon)^2 \epsilon t^{-1/2} + C(C_1\epsilon)^2 \Big(\int_t^{\infty}  \Big\|\frac{[\partial^{I_3}\psi]_{-}}{\langle \tau-r\rangle} \Big\|^2_{L_x^2} \rd \tau\Big)^{1/2} t^{-1/2} \\
& \leq C\epsilon(C_1\epsilon)^2 t^{-1/2}.
\end{align*}
The proof of \eqref{eq:scatterhighorder} is done. Regarding the estimate of \eqref{eq:scatterloworder}, we have 
\begin{align}
	\|(\psi^*\gamma^0\psi)\psi\|_{H^{N-1}} &\leq \sum_{|I|\leq N-1}\|\partial^{I}[(\psi^*\gamma^0\psi)\psi]\|_{L_x^2} \nonumber \\
	& \leq \sum_{\substack{|I_3|+|I_4|\\+|I_2|\leq N-1}} \|[(\partial^{I_3}\psi)^*\gamma^0 \partial^{I_4}\psi]\partial^{I_2}\psi\|_{L_x^2}.
\end{align} 
The slight difference with the estimate of $II$ above is that we can always take $L^2$ norm of the term $[\cdot]_{-}$ and $L^{\infty}$ norm of the other two terms, thus 
\begin{align}
\|(\psi^*\gamma^0\psi)\psi\|_{H^{N-1}}  \lesssim (C_1\epsilon)^2 \sum_{|I_3|\leq N+1} \tau^{-1} \Big\|\frac{[\partial^{I_3}\psi]_{-}}{\langle \tau-r\rangle} \Big\|_{L_x^2},
\end{align}	
from which we finally derive
\begin{align}
	\|\psi(t) - S(t-t_0) \psi^+\|_{H^{N-1}} \leq  C(C_1\epsilon)^2\sum_{|I_3|\leq N+1} \Big(\int_t^{+\infty}  \Big\|\frac{[\partial^{I_3}\psi]_{-}}{\langle \tau-r\rangle} \Big\|^2_{L_x^2} \rd \tau\Big)^{1/2} t^{-1/2}.
\end{align}
Setting 
$$
C(t) = C(C_1\epsilon)^2 \sum_{|I_3|\leq N+1}\Big(\int_t^{+\infty}  \Big\|\frac{[\partial^{I_3}\psi]_{-}}{\langle \tau-r\rangle} \Big\|^2_{L_x^2} \rd \tau\Big)^{1/2}.
$$
Obviously, $\lim_{t\to +\infty} C(t) =0$. This completes the proof.
\end{proof}

\section{Improved estimate for the massless Dirac field}
In this Section, we show the massless Dirac field can gain one order decay of $(t-r)^{-1}$ at the expense of some extra logarithmic growth in time $t$. To begin with, we first state some notations and useful results involved.  Following the notation in \cite{LeMa14}, we introduce the hyperboloidal frame over $\Hcal_s$:
\begin{align*}
\partial_s := \frac{s}{t}\partial_t, \qquad \ \ \overline{\partial}_a := \underline{\partial}_a = \frac{x_a}{t} \partial_t + \partial_a, \ \ a=1,2.
\end{align*}
By a simple computation, one can find the scaling vector field $L_0=t \partial_t + x_1\partial_1 +x_2 \partial_2$ can be rewritten as
\begin{align*}
L_0 = s\partial_s + x^a \overline{\partial}_a, \ \ \ \ x^a= x_a.
\end{align*}



We will make use of the conformal energy, which was first derived by Ma and Huang in \cite{MaHu17} in the context of the hyperboloidal foliation method for wave equations in $\RR^{1+3}$. A similar version in $\RR^{1+2}$ is given in \cite{YM0}, and then is widely used in many contexts, see for instance \cite{DoLeLe21}, where the authors showed the boundedness of top-order energy for two dimensional quasilinear wave equations, and an alternative novel proof is given in \cite{Li21}.  
\begin{proposition} \label{prop:conformalenergy}
Let $u$ be a sufficiently regular function defined in $\Kcal_{[s_0,s_1]}$. Then for all $s\in [s_0,s_1]$, we have 
\begin{align*}
E_{con}^{1/2}(s,u) \leq E^{1/2}_{con}(s_0,u) + \int_{s_0}^s \tau \|\Box u\|_{L_f^2(\Hcal_{\tau})}\, {\rm d}\tau,
\end{align*}
here
 \begin{align}
E_{con}(s,u):= \int_{\Hcal_s} \sum_{a} (s\overline{\partial}_a u)^2 + (Ku+u)^2\, {\rm d}x, \nonumber
\end{align}
where $K:= s\partial_s + 2x_a \overline{\partial}_a$.
\end{proposition}
An important role played by the conformal energy is the control of the $L^2$ type norm for the solution with no derivatives, as proved by Y. Ma in \cite[page 8]{YM0}. We have the following.  
\begin{proposition} \label{prop:l2viaconformal}
Let $u$ be a sufficiently regular function defined in $\Kcal_{[s_0,s_1]}$. Then for all $s\in [s_0,s_1]$, we have
\begin{align}
\Big\|\frac{s}{t} u \Big\|_{L_f^2(\Hcal_s)} \leq \Big\|\frac{s_0}{t} u \Big\|_{L_f^2(\Hcal_{s_0})} + C \int_{s_0}^s \frac{E_{con}^{1/2}(\tau,u)}{\tau}\, \rd \tau. \nonumber
\end{align}
\end{proposition} 
To obtain extra $(t-r)$-decay, we will utilize a classical inequality which shall be formulated as a lemma, and one refers to \cite[page 39]{Sogge} for the proof. 
\begin{lemma}\label{lem:trdecay}
Let $u$ be a sufficiently smooth function, then we have 
\begin{align}
|t-|x|| \sum_{\alpha=0}^2|\partial_{\alpha} u(t,x)| \leq |L_0u(t,x)|+\sum_{a}|L_au(t,x)| + |(\Omega_{1,2}u)(t,x)|. \nonumber
\end{align} 
\end{lemma}
It is worth mentioning that 
\[
\Omega_{1,2}u(t,x)=(x_1\partial_2 - x_2\partial_1)u(t,x) =(x_1\overline{\partial}_2 - x_2\overline{\partial}_1)u(t,x).
\]
Hence for a function $u$ defined in $\Kcal$, we have $|\Omega_{1,2}u| \leq \sum_a |L_au|$. Thanks to Lemma \ref{lem:trdecay}, we can conclude 
\begin{align}\label{eq:betterdecay}
|t-|x||\sum_{\alpha=0}^2 |\partial_\alpha u(t,x)| \leq C\Big(|L_0 u(t,x)| + \sum_a |L_a u(t,x)| \Big), \ \ \ (t,x)\in \Kcal.
\end{align}
This is the formula to be used in what follows.

Back to the cubic massless Dirac equation:
\begin{align}\label{eq:masslessdirac}
i\gamma_{\mu} \partial_{\mu} \psi = (\psi^*\gamma^0\psi)\psi, \ \ \psi(t_0,x)=\psi_0(x).
\end{align}
In  Section \ref{sec:conclusion}, we have proved 
\begin{align*}
|\psi(t,x)| \lesssim \frac{\epsilon}{t^{1/2}(t-r)^{1/2}},
\end{align*}
provided the initial data $\|\psi_0\|_{H^{N+1}} < \epsilon \leq \epsilon_0$ with $N\geq 2$. Our aim in this part is to show that we can further arrive at
\begin{align}
|\psi(t,x)| \lesssim \frac{ (\ln t)^2}{t^{1/2}(t-r)^{3/2}} \epsilon. \nonumber
\end{align}
To proceed, we shall adopt an idea due to Bournaveas \cite{Bo01}. Let $\psi$ be the global solution to \eqref{eq:masslessdirac},  $\Psi$  be chosen to  verify
\begin{align} \label{eq:newpsi}
\Box \Psi = i\gamma^{\mu}\partial_{\mu} \psi = (\psi^*\gamma^0\psi)\psi, \ \ \Psi(t_0,x) =0, \ \  \partial_t\Psi(t_0,x)= -i \gamma^0 \psi_0.
\end{align}
Then $i\gamma^{\mu}\partial_{\mu} \Psi = \psi$, since $i\gamma^{\mu}\partial_{\mu} (i\gamma^{\mu}\partial_{\mu}\Psi-\psi) =0$, $(i\gamma^{\mu}\partial_{\mu}\Psi-\psi)(t_0) =0$. To estimate $\psi$, it suffices to bound $\partial \Psi$. 

Now let us perform the conformal energy estimate for $\Psi$. 
\begin{proposition} \label{prop:boundsonconformal}
Let the assumptions in Theorem \ref{thm:mainresult} hold, and suppose $\Psi$ solves \eqref{eq:newpsi}, then for all $s\in [s_0,\infty)$, we have 
\begin{align*}
	E_{con}^{1/2}(s,\partial^IL^J\Psi) \lesssim  \epsilon \ln s,
	\qquad
	|I| + |J| \leq N+1.
\end{align*}
\end{proposition}
\begin{proof}
Recall that we have proved the Dirac solution $\psi$ satisfies
\[
\Ecal^D(s,\partial^I\widehat{L}^J\psi) \leq C_1 \epsilon, \ \ \ \forall\,s\in [s_0,\infty), \ \ \ |I|+|J| \leq N+1,
\]
provided the associated initial data $\|\psi_0\|_{H^{N+1}} <\epsilon<\epsilon_0$. On the other hand, it is easy to see 
\[
\Box \partial^I L^J\Psi = \partial^IL^J\big[(\psi^*\gamma^0\psi)\psi \big].
\]
Thanks to the local well-posedness theory of  linear wave equation, one can assert there exists some $\widetilde{C}_0>0$, such that 
 \[
 E_{con}^{1/2}(s_0, \partial^IL^J \Psi) \leq \widetilde{C}_0 \epsilon, \ \ \ \forall\, |I|+|J| \leq N+1.
 \]	
 Then in view of Proposition \ref{prop:conformalenergy}, for any $s\in [s_0,\infty)$, we have 
 \begin{align}
 	E_{con}^{1/2}(s, \partial^I L^J \Psi) \leq E_{con}^{1/2}(s_0, \partial^IL^J \Psi) + \int_{s_0}^s \tau \big\|\partial^IL^J\big[(\psi^*\gamma^0\psi)\psi \big] \big\|_{L_f^2(\Hcal_{\tau})}\, \rd \tau.  \nonumber
 \end{align}
  A similar argument as in Proposition \ref{improvebounds} yields
  \[
\big\|\partial^IL^J\big[(\psi^*\gamma^0\psi)\psi \big] \big\|_{L_f^2(\Hcal_{\tau})} \lesssim (C_1 \epsilon)^3 \tau^{-2}.
  \]
  Then one can see
  \[
  E_{con}^{1/2}(s, \partial^I L^J \Psi) \lesssim \epsilon + \epsilon^3 \ln s, \ \ \forall \ |I| + |J| \leq N+1.
  \]
  Noting that $s>s_0=2$, the required result follows.
\end{proof}

As a consequence of the conformal energy estimate, we have the following.
\begin{corollary} \label{cor:l2estimate}
	For all $|I|+|J| \leq N+1$ and $s\in [s_0,\infty)$, there holds
	\begin{align} 
			\Big\|\frac{s}{t}\partial^I L^J \Psi \Big\|_{L_f^2(\Hcal_s)} & \lesssim \epsilon (\ln s)^2, \label{eq:l2estiamteofpsi} \\
		\Big\|\frac{s}{t}L_0 \partial^I L^J \Psi \Big\|_{L_f^2(\Hcal_s)} & \lesssim \epsilon (\ln s)^2. \label{eq:scalingfield}
	\end{align}
\end{corollary}
\begin{proof}
	For any $s\in [s_0,\infty)$, by Proposition \ref{prop:l2viaconformal} and Proposition \ref{prop:boundsonconformal}, one can obtain 
	\begin{equation} \label{eq:weightl2norm}
	\begin{split}
	\Big\|\frac{s}{t}\partial^I L^J \Psi \Big\|_{L_f^2(\Hcal_s)} &\leq \Big\|\frac{s_0}{t}\partial^I L^J \Psi \Big\|_{L_f^2(\Hcal_{s_0})} + \int_{s_0}^s \frac{E^{1/2}_{con}(\tau, \partial^I L^J \Psi)}{\tau}\,\rd \tau \\
	& \lesssim \epsilon + \epsilon\int_{s_0}^s \frac{\ln \tau}{\tau} \,\rd \tau\\
	& \lesssim \epsilon (\ln s)^2.
	\end{split}
	\end{equation}
	Concerning the estimate of \eqref{eq:scalingfield}, we note that $L_0 = s\partial_s + x^a \overline{\partial}_a$ and 
	\begin{align}
		\frac{s}{t} L_0 \partial^I L^J \Psi = \frac{s}{t} (K+1) \partial^I L^J \Psi - \frac{s}{t}  \partial^I L^J \Psi - \frac{s}{t} x^a\overline{\partial}_a \partial^I L^J \Psi.  \nonumber
	\end{align}
	This implies
	\begin{align*}
		&\Big\|\frac{s}{t} L_0 \partial^I L^J \Psi  \Big\|_{L_f^2(\Hcal_s)} \\
		& \leq    \Big\|(K+1)\partial^I L^J \Psi  \Big\|_{L_f^2(\Hcal_s)} + \Big\|\frac{s}{t}\partial^I L^J \Psi \Big\|_{L_f^2(\Hcal_s)}  + \sum_a\Big\|s \overline{\partial}_a\partial^I L^J \Psi \Big\|_{L_f^2(\Hcal_s)} \\
		&\lesssim E^{1/2}_{con}(s,\partial^I L^J\Psi) + \epsilon (\ln s)^2 \\
		&\lesssim \epsilon (\ln s)^2,
	\end{align*}
	where we used Proposition \ref{prop:boundsonconformal} and \eqref{eq:weightl2norm}. The proof is completed.
\end{proof}

Now we are ready to demonstrate the proof of Theorem \ref{thm:improveddecay}.  
\begin{proof}[Proof of Theorem \ref{thm:improveddecay}]
By Corollary \ref{cor:weightl2estimate} and  Corollary \ref{cor:l2estimate}, we can find 
\[
\sup_{\Hcal_s} s |L_0\Psi(t,x)| \lesssim \sum_{|J|\leq 2}\Big\|\frac{s}{t}L_0 L^J \Psi \Big\|_{L_f^2(\Hcal_s)} \lesssim \epsilon (\ln s)^2,
\]
where we also used $[L_0,L_a]=0,\ a=1,2$.  Similarly, 
\begin{align}
	\sup_{\Hcal_s} s|L_a\Psi(t,x)|  \lesssim \epsilon (\ln s)^2. \nonumber
\end{align}
Owing to \eqref{eq:betterdecay} and the fact ${\rm supp}\,\Psi \subset \Kcal$, we finally obtain
\begin{align*}
(t-r)\sum_{\alpha=0}^2\big|	\partial_{\alpha}\Psi(t,x) \big| & \leq C\Big( |L_0 \Psi(t,x)| + \sum_a |L_a \Psi(t,x)| \Big) \\
& \lesssim \epsilon s^{-1} (\ln s)^2.
\end{align*}
This immediately yields 
\begin{align}
	|\psi(t,x)| \lesssim  \frac{ [\ln (t(t-r))]^2}{t^{1/2} (t-r)^{3/2}} \epsilon\lesssim  \frac{ (\ln t)^2}{t^{1/2} (t-r)^{3/2}} \epsilon  , \nonumber
\end{align}
given that $i\gamma^{\mu} \partial_{\mu} \Psi = \psi$ and $s\sim t^{1/2}(t-r)^{1/2}$. Combining the decay result in Theorem \ref{thm:mainresult}, we see 
\begin{align*}
	|\psi(t,x)| \lesssim  \min \Big\{\frac{1}{t^{1/2}(t-r)^{1/2}}, \frac{ (\ln t)^2}{t^{1/2} (t-r)^{3/2}} \Big\}\epsilon.
\end{align*}
This concludes the proof.
\end{proof}

 \section*{Acknowledgments} 
Both authors are grateful to Prof. Zhen Lei (Fudan University) for his constant support and encouragement.
The author S.D. owes thanks to Dr. Zoe Wyatt (Cambrige University) for many discussions.
The authors would also like to thank Dr. Yunlong Zang (Yangzhou University) for many discussions.
The author S.D. was partially supported by the China Postdoctoral Science Foundation, with grant number 2021M690702.



\end{document}